\documentclass{amsart}
\usepackage{amsfonts,amssymb,verbatim,amsmath,amsthm,latexsym,textcomp,amscd}
\usepackage{latexsym,amsfonts,amssymb,epsfig,verbatim}
\usepackage{amsmath,amsthm,amssymb,latexsym,graphics,textcomp, hyperref}
\usepackage{graphicx}
\usepackage[capitalize]{cleveref}
\usepackage{color}
\usepackage{url}
\usepackage{psfrag}
\usepackage{amsbsy}
\usepackage{bm}
\usepackage{enumerate}
\newtheorem{theorem}{Theorem}[section]

\newtheorem{definition}[theorem]{Definition}
\newtheorem{assumption}[theorem]{Assumption}
	\newtheorem{prop}[theorem]{Proposition}
	\newtheorem{proposition}[theorem]{Proposition}
	\newtheorem{lemma}[theorem]{Lemma}

	\newtheorem{qn}[theorem]{Question}
	\theoremstyle{definition}

	\newtheorem{rmk}[theorem]{Remark}

\newcommand{\hyperbolic}{\mathbb{H}}
\newcommand{\Ga}{\Gamma}
\newcommand{\reals}{\mathbb{R}}

\newcommand{\complexes}{\mathbb{C}}

\newcommand{\Fr}{\operatorname{Fr}}
\newcommand{\integers}{\mathbb{Z}}
\newcommand{\naturals}{\mathbb{N}}

\newcommand{\pslc}{\operatorname{PSL}_2\complexes}

\newcommand{\ie}{i.e.\ }

	\newcommand{\G}{{\Gamma}}

	\newcommand{\natls}{{\mathbb N}}

	\newcommand\HHH{{\mathbb H}}

	\newcommand\EE{{\mathcal E}}
	\newcommand\FF{{\mathcal F}}
	
	\newcommand\HH{{\mathcal H}}

	\newcommand\LL{{\mathcal L}}
	\newcommand\MM{{\mathcal M}}

	\newcommand\PP{{\mathcal P}}
	
	\newcommand\RR{{\mathcal R}}

	\newcommand\PMF{{\PP\kern-2pt\MM\FF}}
	
	\newcommand\PML{{\PP\kern-2pt\MM\LL}}

	\newcommand\ep{\epsilon}

	\newcommand\Hyp{{\mathbb H}}
	
	\newcommand\Z{{\mathbb Z}}

	\newcommand\EXH{{ \EE (X, \HH )}}

	\newcommand\til{\widetilde}
	\newcommand\length{\operatorname{length}}

	\newcommand{\pr}{\mathrm{pr}}

	\newcommand{\fsubd}{\mathrel{{\scriptstyle\searrow}\kern-1ex^d\kern0.5ex}}
	\newcommand{\bsubd}{\mathrel{{\scriptstyle\swarrow}\kern-1.6ex^d\kern0.8ex}}
	\newcommand{\fsubeq}{\mathrel{\raise-.7ex\hbox{$\overset{\searrow}{=}$}}}
	\newcommand{\bsubeq}{\mathrel{\raise-.7ex\hbox{$\overset{\swarrow}{=}$}}}

	\newcommand{\tsh}[1]{\left\{\kern-.9ex\left\{#1\right\}\kern-.9ex\right\}}

\begin{document}
	\title{Discontinuous motions of limit sets}
	\author{Mahan Mj}
	\address{School
		of Mathematics, Tata Institute of Fundamental Research\\ 1, Homi Bhabha Road, Mumbai-400005, India}
	
	\email{mahan@math.tifr.res.in}
	
	\author{Ken'ichi Ohshika}
		\address{Department of Mathematics, Gakushuin University, Mejiro, Toshima-ku, Tokyo, Japan}
		\email{ohshika@math.gakushuin.ac.jp}
	
	\date{\today}
	
	\thanks{MM is   supported by  the Department of Atomic Energy, Government of India, under project no.12-R\&D-TFR-14001.
		MM is also supported in part by a Department of Science and Technology JC Bose Fellowship, CEFIPRA  project No. 5801-1, a SERB grant MTR/2017/000513, and an endowment of the Infosys Foundation via the Chandrasekharan-Infosys Virtual Centre for Random Geometry.   This material is based upon work partially supported by the National Science Foundation
		under Grant No. DMS-1928930 while MM participated in a program hosted
		by the Mathematical Sciences Research Institute in Berkeley, California, during the
		Fall 2020 semester.
		KO is supported by the JSPS Grant-in-Aid for Scientific Research (B)  17H02843}
	
	\begin{abstract}
		We characterise completely when limit sets, as parametrised by Cannon-Thurston maps, move discontinuously for a sequence of algebraically  convergent quasi-Fuchsian groups.
	\end{abstract}

	\maketitle

	\tableofcontents

\sloppy
\section{Introduction}
In [Question 14] of  \cite{thurston-bams}, Thurston raised a question about continuous motions of limit sets under algebraic deformations of Kleinian groups. This problem was formulated more precisely in \cite{mahan-series1, mahan-series2} taking into account topologies of convergence of Kleinian groups and a parametrisation of limit sets using  Cannon-Thurston maps \cite{CT, mahan-split}.
The questions can be stated as follows.

\begin{qn}\label{mainq}  \cite{thurston-bams, mahan-series1, mahan-series2}\\
	\begin{enumerate}
		\item[(1)] If a sequence of isomorphic Kleinian groups $(G_n)$ converges to $G_\infty$ algebraically then do the corresponding Cannon-Thurston maps converge pointwise?
		
		\item[(2)] If $(G_n)$ converges to $G_\infty$ strongly then do the corresponding Cannon-Thurston maps converge uniformly?
	\end{enumerate}
\end{qn}

It was established,  in \cite{mahan-series2} and \cite{mahan-cmsurvey} (crucially using technology developed in \cite{mahan-split}, \cite{mahan-kl}) that the second question has an affirmative answer. The answer to the first question has turned out to be considerably subtler. Indeed, interesting examples of both continuity and discontinuity occur naturally:

\begin{enumerate}
	\item In \cite{mahan-series1}, it was shown that if the geometric limit  $\Gamma$ of $(G_n)$ is geometrically finite, then the answer to Question \ref{mainq} (1) is affirmative. In particular, this is true for the examples given by Kerckhoff-Thurston in \cite{kerckhoff-thurston}.
	\item On the other hand, it was shown in \cite{mahan-series2}, that for certain examples of quasi-Fuchsian groups converging to geometrically infinite groups constructed by Brock \cite{brock-itn}, the answer to Question \ref{mainq} (1) is negative.
\end{enumerate}

In this paper, we shall complete the answer to Question \ref{mainq} (1) for sequences of quasi-Fuchsian groups  by characterising precisely when  limit sets move discontinuously, i.e. we isolate the fairly delicate conditions that ensure the discontinuity phenomena illustrated in \cite{mahan-series2} for Brock's examples. 
In order to do this, we need to understand possible geometric limits of sequences of Kleinian surface groups. The necessary technology was developed in \cite{OS} and \cite{OhD}.
In particular, it was shown there that there exists an embedding  of any such geometric limit into $S \times (-1,1)$. 
In the following, we shall mainly refer to \cite{OhD}, where a simplified proof for convergent sequences is given, rather than to \cite{OS}.
 
 Let $(\rho_n: \pi_1(S) \rightarrow \pslc)$ be a sequence of quasi-Fuchsian groups  converging algebraically to $\rho_\infty: \pi_1(S) \rightarrow \pslc$, where $S$ is a hyperbolic surface of finite area.
 We set $G_n=\rho_n(\pi_1(S))$ and $G_\infty=\rho_\infty(\pi_1(S))$.
 Suppose that $(G_n)$ converges geometrically to a Kleinian group $\Gamma$.
 In what follows, we need to consider ends of the non-cuspidal part $(\HHH^3/G_\infty)_0$ (resp. $(\HHH^3/\Gamma)_0$) of the algebraic (resp. geometric) limit.
For an end $e$ of either $(\HHH^3/G_\infty)_0$ or $(\HHH^3/\Gamma)_0$, if  there is a $\integers$-cusp neighbourhood whose boundary $A$ has an end tending to $e$, we say that $A$ (or the corresponding $\integers$-cusp neighbourhood) {\bf abuts} on $e$.
Abusing terminology, we also say that for a neighbourhood $E$ of $e$  the annulus $A$ or the corresponding cusp neighbourhood abuts on $E$  in this situation. Before stating the main theorem of this paper, we shall introduce some terminology.

We first explain what it means for a simply degenerate end $e$ of $(\HHH^3/G_\infty)_0$ to be {\bf coupled}. See figure below.
We shall describe this condition more precisely in \S \ref{ends}, and we just give a sketch here. Let $E$ be a neighbourhood of $e$ homeomorphic to $\Sigma \times \reals$ for an essential subsurface $\Sigma$ of $S$. The neighbourhood $E$ can be taken to project homeomorphically to a neighbourhood $\bar E$ of an end $\bar e$ of $(\HHH^3/\Gamma)_0$, the non-cuspidal part of the geometric limit.  The end $\bar e$ is required to satisfy the following.
There is another end $\bar e'$ of $(\HHH^3/\Gamma)_0$,  simply degenerate or wild,  called a {\bf partner} of $\bar e$. The ends $\bar e'$ and $\bar e$ are related  as follows. The end $\bar e'$ has a neighbourhood $\bar E'$ such that 
if we pull back   $\bar E$ and $\bar E'$ by approximate isometries to $\HHH^3/G_n$ for large $n$, then their images are both contained in a submanifold of the form $\Sigma \times (0,1)$, where $\Fr \Sigma \times (0,1)$ lies on the boundaries of Margulis tubes, and $\Sigma \times \{t\}$ is  incompressible. \cref{M_n} is a schematic representation of $\HHH^3/G_n$ and the second is a schematic representation of the limit $(\HHH^3/\Gamma)_0$.
Note that   $\bar E'$  is not required to be homeomorphic to $\Sigma \times \reals$, as indicated by the semicircular bump in the middle of $\bar E'$ in \cref{Gamma}.
In fact   $\bar E'$  may only be homeomorphic to $\Sigma' \times \reals$ for some essential subsurface $\Sigma'$  of $S$, which shares a boundary component with $\Sigma$.

\begin{figure}

	\includegraphics[height=5cm]{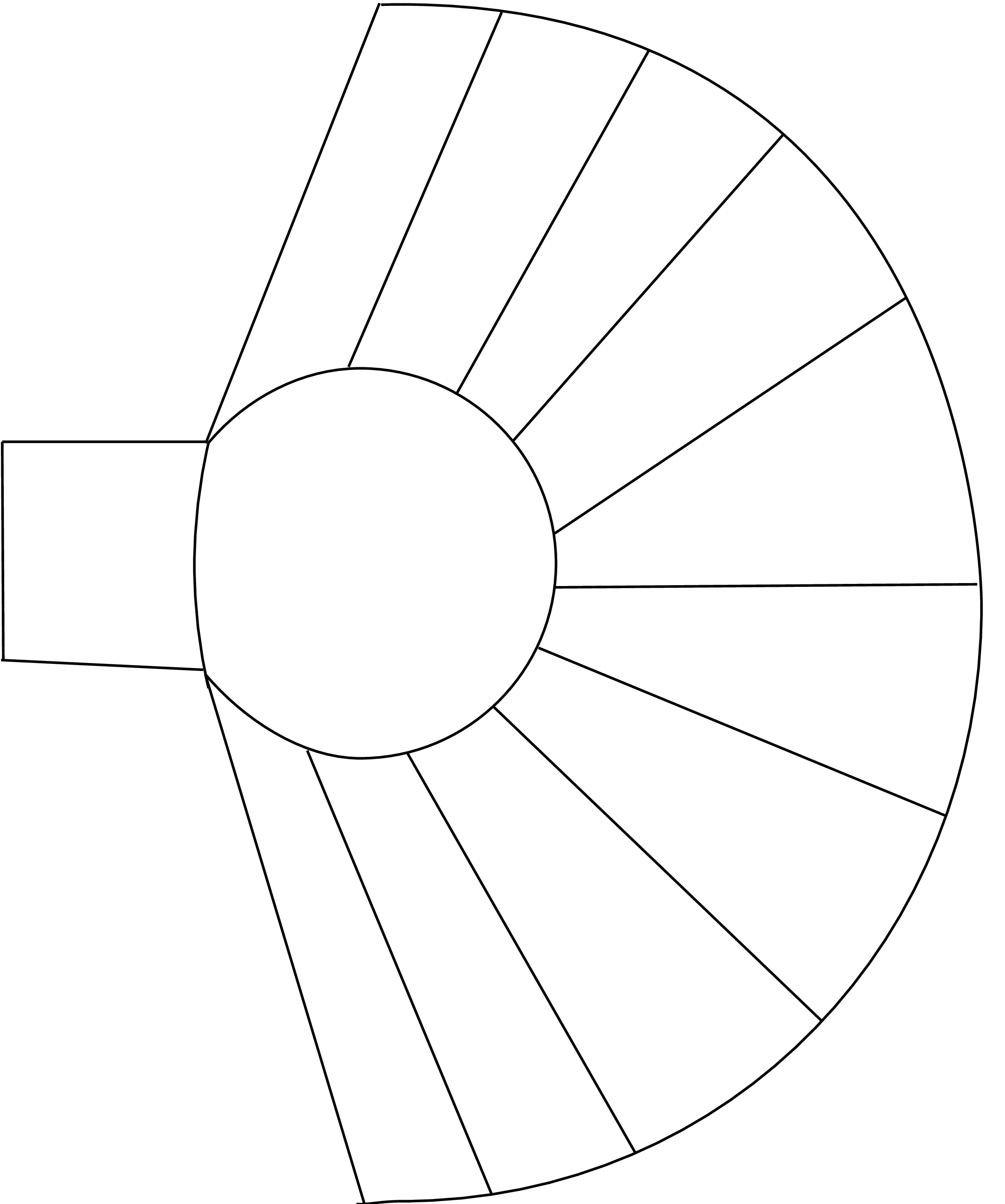}

	\caption{Schematic representation of $\HHH^3/G_n$}
	\label{M_n}
	\end{figure}

\begin{figure}

\includegraphics[height=5cm]{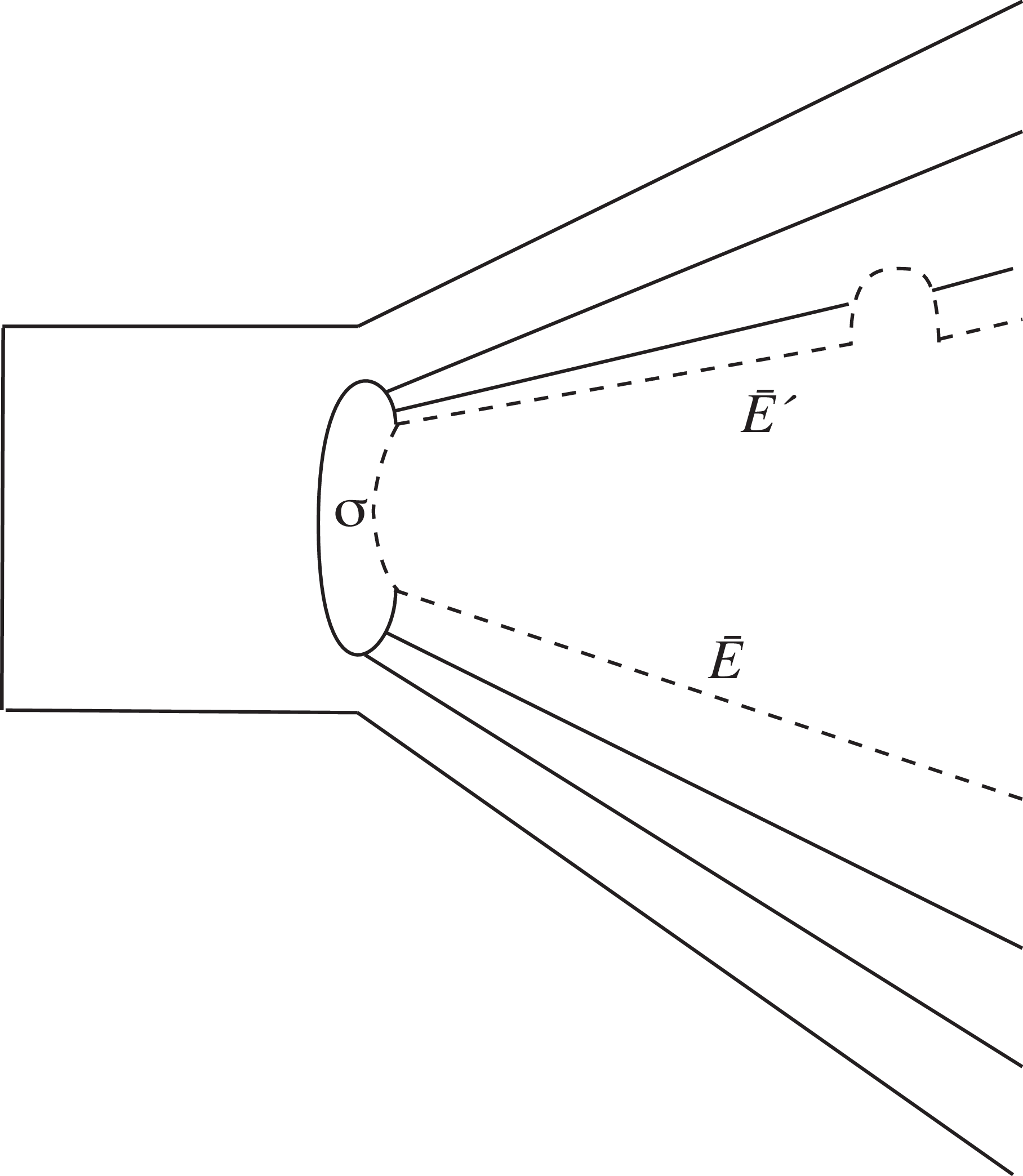}

\caption{Schematic representation of coupled ends $\bar E, \bar E'$ in $(\HHH^3/\Gamma)_0$ with conjoining cusp corresponding to $\sigma$}
\label{Gamma}
\end{figure}

  A $\integers$-cusp corresponding to a parabolic curve $\sigma$, abutting on both an end $\bar E$ and its partner
   $\bar E'$ is said to {\bf conjoin}  $\bar E$ and 
   $\bar E'$ or $e, e'$. Such a cusp will sometimes be referred to simply as a {\bf conjoining} cusp (see \cref{Gamma}).
  Such a conjoining cusp is called {\bf untwisted} if  the Dehn twist parameters of the Margulis tubes corresponding to the $\integers$-cusp are bounded uniformly all along the sequence $(\rho_n)$.
  
  Let $\mu$ be the ending lamination of the end $\bar e$.
  A  crown domain for $(\mu, \sigma)$  is said to be {\bf well-approximated} if 
    it is realised  in $\bar E'$ (with respect to the marking determined by approximate isometries).  
  See \S\ref{ends} for a more precise definition. We now state the main theorem of this paper.

\begin{theorem}
	\label{introthm1}{\rm
	Let $c_n : S^1 (=\Lambda_{\pi_1(S)}) \rightarrow \Lambda_{\rho_n(\pi_1(S))}$ denote
 the Cannon-Thurston maps  for the quasi-Fuchsian representations $\rho_n\ (n= 1, \cdots, \infty)$. Then $(c_n)$ does not converge pointwise to   $c_\infty$ if and only if all of Conditions (1)-(3) below are satisfied.
 
 \begin{enumerate}[{\rm ({\text Condition} 1:)}]
 	\item The algebraic limit has a   coupled end $e$ with a partner $e'$.
%	 We call $e'$ the {\bf partner} of $e$, and a $\integers$-cusp abutting both $\bar E$ and $\bar E'$  {\bf conjoining}.
%	Identifying $\HHH^3/\Gamma$ with the image of its embedding into $S \times (-1,1)$, there is
% 	$t_0 \in (-1,1)$ and a proper essential subsurface $\Sigma$ of $S$ such that	$\Sigma \times \{ t_0 \}$ is not contained in the image of 
% 	of the geometric limit under $\Phi$.
 	%Further,  the image under $\Phi$  has two ends of the form $\Sigma \times [t_0-\epsilon, t_0)$ and  $\Sigma \times (t_0, t_0+\epsilon]$ for some small $\epsilon > 0$.
 	
% 	\item[(Condition 2:)] There exists an  open annulus $A$ on the boundary of $E$ which also abuts  a geometrically finite end of $(\HHH^3/G_\infty)_0$.
% 	We shall call such a $\sigma$ {\bf finite-separating}.
 	
 	\item
	There is an untwisted cusp $\integers$-cusp corresponding to a parabolic curve $\sigma$, conjoining $e, e'$. 
 	
 	\item Let $\mu$ be the ending lamination of the end $\bar e$. There exists a well-approximated crown domain for $(\mu, \sigma)$.
	%We shall say that $\sigma$ (or equivalently the cusp corresponding to $\sigma$) {\bf abuts} $\mu$.
 \end{enumerate}}
\end{theorem}

\begin{rmk}{\rm
Condition 1 is concerned with the type of geometric limit.  However, Conditions 2 and 3 deal with the} {\it manner} {\rm in which the sequence converges geometrically.

In fact Theorem \ref{introthm1} shows that we  can construct examples of two sequences of quasi-Fuchsian groups having the same geometric limit and also the same algebraic limit; however, the sequence of  Cannon-Thurston maps have radically different behaviour. For one sequence, the Cannon-Thurston maps converge pointwise, for the other they fail to do so.

It is this subtlety that is captured by the somewhat technical nature of the statement of Theorem \ref{introthm1} above.}
\end{rmk}

We shall also characterise the points where this discontinuity occurs (see Theorem \ref{non-continuous}). These turn out to be exactly the tips of crown domains as in \cite{mahan-series2}.
	
\begin{theorem}
	\label{introthm2}{\rm
	In the setting of Theorem \ref{introthm1}, suppose that
	\begin{enumerate}
	\item  $(\HHH^3/G_\infty)_0$ has a coupled simply degenerate end $e$ with  ending lamination $\lambda$,
	\item an untwisted $\integers$-cusp corresponding to a parabolic curve $s$ abutting on the image of $e$ in $(\HHH^3/\Gamma)_0$,
	\item  there exists a well approximated crown domain for $(\lambda,s)$.
	\end{enumerate}
	 
	Then for $\zeta \in S^1$, the images $c_n(\zeta)$ do not converge to $c_\infty(\zeta)$ if and only if $\zeta$ is a tip of a crown domain $C$ for $(\mu, \sigma)$, where 
	
	\begin{enumerate}
		\item
		$\mu \cup \sigma$ is contained  in either the union of the lower parabolic curves and the lower ending laminations (denoted by $\lambda_-$), or the union of  the upper parabolic curves and the upper ending laminations, (denoted by  $\lambda_+$);
		\item 
		the crown domain $C$ is well approximated; 
		\item
		the simple closed curve $\sigma$  corresponds to an  untwisted conjoining cusp abutting on the projection of the end  in $(\HHH^3/\Gamma)_0$ for which $\mu$ is the ending lamination.
	\end{enumerate}}
	\end{theorem}

\section{Preliminaries}
\subsection{Some basic material for hyperbolic 3-manifolds}
Let $(S,m)$ be a complete hyperbolic surface of finite area, possibly with cusps.
A geodesic lamination on $(S,m)$ is a closed subset of $S$ which is a disjoint union of simple geodesics.
The notion of geodesic lamination depends on the hyperbolic metric $m$.
However, given any geodesic lamination $\lambda$ on $(S,m)$ and any  complete hyperbolic metric $m'$ on $S$, there is a unique geodesic lamination $\lambda'$ on $(S,m')$ which is ambient isotopic to $\lambda$.

For a  surface $S$ as above and a hyperbolic 3-manifold $M$, a continuous $f \colon S \to M$ is said to be a pleated surface if there is a complete hyperbolic metric $m$ on $S$ and a geodesic lamination $\lambda$ on
$(S,m)$ such that $f$ maps each leaf of $\lambda$ to a geodesic, and each component of $S \setminus \lambda$ totally geodesically. The geodesic lamination $\lambda$ is called the pleating locus of the pleated surface
$f \colon S \to M$.
A geodesic lamination $\lambda$ on $S$ (with some fixed complete hyperbolic metric) is said to be {\bf realised} by a pleated surface $f \colon S \to M$ if 
$\lambda$ is ambient isotopic to the pleating locus of the pleated surface
$f \colon S \to M$.

%simply degenerate end, ending lamination

\subsection{Relative hyperbolicity and electrocution} 
\label{rel hyp}
We refer the reader to \cite{farb-relhyp} and \cite{bowditch-relhyp} for generalities on relative hyperbolicity and to \cite{mahan-ibdd} and \cite{mahan-split} for the notions of electrocution and electro-ambient paths. 
We shall briefly recall the notion of electro-ambient quasi-geodesics, c.f.\ \cite{mahan-split}.
Let $(X,d_X)$ be a $\delta$-hyperbolic metric
space. 
Bowditch showed in \cite{bowditch-relhyp} that if there are constants $C, D$ and a family $\mathcal K$ of $D$-separated, $C$-quasi-convex sets in $X$, then $X$ is (weakly) hyperbolic relative to $\mathcal K$.
Now let $\mathcal H$ be a collection of $C$-quasi-convex sets in $(X, d_X)$, without assuming the $D$-separated condition.
 Let $\EXH$ denote the  space obtained by electrocuting the elements of $\mathcal H$ in $X$: this space is a union of $X$ and $\sqcup_{H\in \mathcal H} H \times [0,1/2]$, where $H \times \{0\}$ is identified with $H$ in $X$, each $\{h\} \times [0,1/2]$ is isometric to $ [0,1/2]$, and $H \times \{1/2\}$ is equipped with the zero metric.
Since $\{H \times \{1/2\}\}$ is $1$-separated, we can apply Bowditch's result, and see that $\EXH$ is Gromov hyperbolic.

Let $\alpha = [a,b]$ be a geodesic in $(X,d_X)$, and 
$\beta $ 
an electric 
quasi-geodesic without backtracking 
joining $a, b$ in $\EXH$, i.e.\ an electric quasi-geodesic which does not return to an element $H \in \HH$ after leaving it. 
We further assume that the intersection of $\beta$ and $H \times (0,1/2)$ is either empty or a disjoint union of open arcs of the form $\{h\} \times (0,1/2)$.
We parametrise $\beta$, and consider the maximal subsegments of $\beta$
contained entirely in some $H \times \{1/2\}$ (for  some $H \in \mathcal{H}$). 
We extend each  such maximal subsegment  by adjoining `vertical' subsegments (of the form $h \times [0, 1/2])$ in $\beta$ at  its endpoints
to obtain a path of the form $\{ p \} \times [0,1/2]\cup [p, q]\times\{1/2\}
 \cup \{ q \} \times [0,1/2]$.
 We call 
these subpaths of $\beta$ {\it  extended  maximal subsegments}.
We replace each extended  maximal subsegment  in $\beta$
by a   geodesic path in $(X,d_X)$ joining the same endpoints.

The resulting
path $\beta_q$ is called an {\bf electro-ambient representative} of $\beta$ in
$X$. Also, if $\beta $ is
an electric  $P$-quasi-geodesic 
%(resp. $(K,\epsilon)$-quasi-geodesic) 
without backtracking (in $\EE(X,\HH )$), then $\beta_q$ is called an {\bf electro-ambient $P$-quasi-geodesic}.
%  (resp. electro-ambient 
%$(K,\epsilon)$-quasi-geodesic).
If  $\beta $ is
an electric  geodesic  without backtracking, then $\beta_q$ is simply called an {\bf electro-ambient
	quasi-geodesic}.
	The following lemma says that hyperbolic geodesics do not go far from electro-ambient quasi-geodesic realisations.

\begin{lemma}{\rm (\cite[Proposition 4.3]{klarreich}, \cite[Lemma
	3.10]{mahan-ibdd} and \cite[Lemma 2.5]{mahan-split})}
		For given non-negative numbers $\delta$, $C$ and $P$,  there exists $R$ such that the following
	holds: \\
	Let $(X,d_X)$ be a $\delta$-hyperbolic metric space and $\mathcal{H}$ a
	family of $C$-quasi-convex
	subsets of $X$. Let $(\EE(X,\HH ),d_e)$ denote the electric space obtained by
	electrocuting the elements of $\mathcal{H}$.  Then, $(\EE(X,\HH ),d_e)$ is Gromov hyperbolic, and  if $\alpha , \beta_q$
	denote respectively a geodesic arc with respect to $d_X$, and an electro-ambient
	$P$-quasi-geodesic with the same endpoints in $X$, then $\alpha$ lies in the	$R$-neighbourhood of $\beta_q$ with respect to $d_X$.
	\label{ea-strong}
	\end{lemma}

\subsection{Cannon-Thurston Maps} We shall review known facts about Cannon-Thurston maps focusing on the case of interest in this paper.
Let $(Y,{d_Y})$ be a Cayley graph of $\pi_1(S)$ for $S$ a closed surface of genus at least 2 with respect to some finite generating system, and
set $X = \Hyp^3$.
By
adjoining the Gromov boundaries $\partial{X} (=S^2)$ and $\partial{Y}  (=S^1)$
to $X$ and $Y$ respectively, we obtain their compactifications  $\widehat{X}$ and $\widehat{Y}$ respectively.

Suppose that $\pi_1(S)$ acts on $\HHH^3$ by isometries as a Kleinian group $G$ via an isomorphism $\rho \colon \pi_1(S) \to G$, and  let $ i :Y \rightarrow X$ be a $\pi_1(S)$-equivariant injection.

\begin{definition}  
	A {\bf Cannon-Thurston map}   $\hat{i}$ (for $\rho$)  from $\widehat{Y}$ to
	$\widehat{X}$ is a continuous extension of $i$.
\end{definition}
The image of $\hat{i}$ restricted to $\partial Y$ coincides with the limit set of $G$.
It is easy to see if a Cannon-Thurston map exists, it is unique.
The notion of Cannon-Thurston map can be easily extended to the case where $S$ is a  hyperbolic surface of finite area. In this situation, it is a $\pi_1(S)$-equivariant continuous map from the relative (or Bowditch) boundary relative to the cusp subgroups,  $\partial_\infty \til{S} = \partial_\infty \Hyp^2  =S^1$, onto the limit set in $S^2$. 
 The first author \cite{mahan-elct} showed that for any Kleinian group isomorphic to a surface group (possibly with punctures), a Cannon-Thurston map always exists, and gave the following characterisation of non-injective points.
 Recall that  an isomorphism from  a Kleinian group  to another Kleinian group is said to be weakly type-preserving when every parabolic element is sent to a parabolic element.

\begin{theorem} \cite{mahan-elct}  \label{ptpre-ct} Let $S= \HHH^2/F$ be a (possibly punctured) hyperbolic surface of finite area. 
Let $\rho : F \rightarrow \pslc$ be a weakly type-preserving   discrete faithful representation with  image $G$.
Let $\lambda_1$ be the union of parabolic curves and ending laminations for upper ends and $\lambda_2$  that of the lower ends, one (or both) of which might be empty.
We regard $\lambda_1$ and $\lambda_2$ as geodesic laminations on $S$.
%Let $M= \HHH^3/G$ be a simply  
%	degenerate hyperbolic manifold
%	corresponding to a faithful, discrete, type-preserving representation $\rho : F \to G$.
%	Suppose that $M$ has  ending lamination $\lambda$ 
%	and let $\hat i : \partial_\infty \HHH^3 \to \Lambda_G$ be the corresponding $CT$-map. Then $\hat i (\xi) = \hat i (\eta)$ for $\xi,\eta \in   \partial_\infty \HHH^2 $ if and only if 
%	$\xi$ and $\eta$ are  either ideal 
%	endpoints of the same  leaf of $\tilde{\lambda}$,  or ideal vertices of a 
%	complementary ideal polygon of  $\tilde{\lambda}$, where $\tilde \lambda$ denotes the preimage of $\lambda$ in the universal covering $\HHH^2$ of $S$. 
%
%If, on the other hand, 
% $M$ is doubly degenerate with  ending laminations $\lambda_k$ ($k=1,2$), then

For $k=1,2$, let $\RR_k$ denote the  relation on $\partial_\infty \HHH^2$  defined as follows: $\xi \RR_k \eta$ if and only if  $\xi$ and $\eta$ are  either ideal 
 endpoints of the same  leaf of $ \tilde \lambda_k$,  or ideal boundary points of a 
 complementary ideal polygon of  $\tilde \lambda_k$, where $\tilde \lambda_k$ is the preimage of $\lambda_k$ in $\HHH^2$. 
  Denote the  transitive closure of $\RR_1\cup \RR_2$ by $\RR$.
Let $\hat i_\rho : \partial {F} \to \Lambda_G$ be the Cannon-Thurston-map for $\rho$. Then $\hat i (\xi) = \hat i (\eta)$ for $\xi,\eta \in   \partial_\infty \HHH^2 $ if and only if 
$\xi \RR \eta$.  \end{theorem}
	
\begin{rmk}
When $S$ has punctures, $\RR$ is strictly larger than $\RR_1 \cup \RR_2$.
In fact, a vertex of a complementary ideal polygon of $\tilde \lambda_1$ containing a lift $p$ of a puncture is related by $\RR$, but not by $\RR_1 \cup \RR_2$, to any vertex of a complementary ideal polygon of $\tilde \lambda_2$ containing $p$.
\end{rmk}

\subsection{Algebraic and Geometric Limits}

Let $(\rho_n : G \rightarrow \pslc)$ be a sequence of weakly type-preserving, discrete, faithful
representations of a fixed finitely generated torsion-free
group $G$
converging to a discrete, faithful representation $\rho_\infty \colon G \rightarrow \pslc$. Also
assume that $( \rho_n(G))$ converges 
to a  Kleinian group
$\Gamma$ as a sequence of closed subsets of $\pslc$ in the Hausdorff topology. Then $\rho_\infty(G )$ is called the 
{\bf algebraic limit} of the sequence and $\Gamma$
the {\bf geometric limit} of the sequence $(\rho_n(G))$.
If $\rho_\infty(G)=\Gamma$, we say that the limit is {\bf strong}.
We note that throughout this paper, when we talk about algebraic and geometric limits, {\em we consider a sequence of representation, and not a sequence of conjugacy classes of representations}.

There is a more geometric way to think  of geometric limits (see \cite{thurstonnotes} and \cite{CEG}). 
A sequence of manifolds
	with basepoints $\{(N_i, x_i )\}$ is said to converge  geometrically 
	to a manifold with basepoint $(N, x_\infty)$ if for any $R>0$ and $K>1$, there exist $i_0(R,K)>0$ and compact submanifolds $C_i\subset N_i$ and $C\subset N$ containing  $R$-balls around $x_i$ and $x_\infty$ respectively, such that there exist  $K$-bi-Lipschitz maps 
	$h_i\colon C_i \to C$  for all $i \geq i_0(R)$.
A sequence of Kleinian groups $(G_n)$ converges geometrically to $\Gamma$ if and only if for a fixed basepoint $x \in \HHH^3$ and its projections $x_n$ and $x_\infty$ in $\HHH^3/G_n$ and $\HHH^3/\Gamma$, the sequence $\{(\HHH^3/G_n, x_n)\}$ converges geometrically to $(\HHH^3/\Gamma, x_\infty)$.
\medskip

\subsection{Criteria for Uniform/Pointwise convergence}
We recall some material from \cite{mahan-series1, mahan-series2}.
  Let $G$ be a fixed finitely generated
	Kleinian group, and $(\rho_n (G) = G_n)$  a weakly type-preserving sequence
	of  Kleinian groups converging algebraically to  $ 
	G_\infty = \rho_\infty (G)$.  Also fix a basepoint $o_{\HHH^3} \in \Hyp^3$. Let $d_G$ denote
	the distance in a Cayley graph of $G$ and $d$ 
	the distance in $\Hyp^3$. Also $[g,h]$ denotes a geodesic in $G$ joining $g$ with $h$, and 
	$[\rho_n(g) (o_{\HHH^3}), \rho_n(h) (o_{\HHH^3})]$ denotes a geodesic in $\HHH^3$ joining  $\rho_n(g) (o_{\HHH^3})$ with $\rho_n(h) (o_{\HHH^3})$.

	\begin{definition} \label{uep}
	The sequence 
	$(\rho_n)$ is said to have the
	{\bf  Uniform Embedding of Points} property (UEP for short)
	if there
	exists a non-negative function  $f(N)$, with
	$f(N)\rightarrow\infty$ as $N\rightarrow\infty$, such that for all $g
	\in \Ga$, 
	$d_\Ga (1,g) \geq N$ implies  $d (\rho_n(g) (o_{\HHH^3}) , o_{\HHH^3}) \geq
	f(N)$ for all $n = 1, \cdots, \infty$. 
	
	The sequence 
	$(\rho_n)$ is said to have the
	{\bf  Uniform Embedding of Pairs of Points} property (UEPP for short)
	if there
	exists a non-negative function  $f(N)$, with
	$f(N)\rightarrow\infty$ as $N\rightarrow\infty$, such that for all $g, h
	\in \Ga$, $d_\Ga (1,[g,h] ) \geq N$ implies  $d ([\rho_n(g) (o_{\HHH^3}), \rho_n(h) (o_{\HHH^3})] , o_{\HHH^3}) \geq
	f(N)$ for all $n = 1, \ldots, \infty$.
\end{definition}

The property UEP is used in \cite{mahan-series1} to give a sufficient criterion to ensure that
algebraic convergence is also geometric. The property UEPP is used to give the following 
criterion for proving uniform convergence of Cannon-Thurston maps.

\begin{prop}[\cite{mahan-series1}]\label{unifcrit1} Let $\Ga$ be a geometrically finite Kleinian group  and let $\rho_n: \Ga \to G_n$ be  weakly type-preserving  isomorphisms to Kleinian groups. 
	Suppose that $(\rho_n)$ converges algebraically to a 
	representation $\rho_{\infty}$. If   $(\rho_n)$ satisfies UEPP, the corresponding Cannon-Thurston maps  
	converge uniformly. 
\end{prop}

\noindent {\bf Notation:} We shall henceforth fix a complete hyperbolic structure of finite area  on $S$ and
 a Fuchsian group $G$ corresponding to the hyperbolic structure.
The limit set $\Lambda_G$ is homeomorphic to $S^1$.
Similarly, 
 $\Lambda_{G_n}$ will denote the limit set of $G_n$ setting $G_n = \rho_n(G)$.
For each $G_n$ $(n \in \naturals \text{ or } n=\infty)$, we shall denote the corresponding Cannon-Thurston map by $c_n : S^1=\Lambda_G \rightarrow \Lambda_{G_n}$.

For pointwise convergence of Kleinian surface groups, a weaker condition called  EPP is sufficient.
This condition depends on  points of $\Lambda_G$.

\begin{prop}[\cite{mahan-series1}] \label{ptwisecrit}
	Let $G \cong \pi_1(S)$ be a Fuchsian group corresponding to a hyperbolic surface $S$ of finite area. 
	%Let $\partial G (=S^1)$ denote the  boundary of $\til S$ (or equivalently $\Ga$). 
	Take $\xi \in \Lambda_G=S^1_\infty$
	and let $[o_{\HHH^2},\xi )$ be a geodesic ray in $\HHH^2$ from a fixed basepoint $o_{\HHH^2}$ to $\xi$. Let $o_{\HHH^3} \in \Hyp^3$
	be a fixed basepoint in $\HHH^3$. 
	\begin{enumerate}
	\item Suppose that $(\rho_n : G \rightarrow \pslc)$ is a sequence
	of weakly type-preserving discrete faithful representations converging  algebraically to $\rho_\infty : G \rightarrow \pslc$.
	Set $G_n = \rho_n (G) , n = 1, \cdots , \infty$ and $M_n = \Hyp^3 / G_n$. 
	\item Let $\phi_n: S \to M_n$ be an incompressible embedding inducing $\rho_n$ at the level of fundamental groups. Let  $\Phi_n: \HHH^2\to \HHH^3$ be a lift of $\phi_n$ (in particular, $\Phi_n$ is an embedding).
	% the (vertex set of the) Cayley graph $\Ga$
	%given by $\rho_n (g) = \rho_n (g) (o_{\HHH^3})$.
	\end{enumerate}

	Then 
	the Cannon-Thurston maps for  the $\rho_n$ converge  to the Cannon-Thurston map for 
	$\rho_\infty$ at $\xi$ if \\
%	\begin{enumerate}
%		\item {\bf EP:} There 
%		exists a non-negative function  $f(N)$, such that 
%		$f(N)\rightarrow\infty$ as $N\rightarrow\infty$ and for all $g \in [1,\xi)$
%		lying outside $B_N(1)$, $\rho_n(g).0$   lies outside the $f(N)$-ball around 
%		$0 \in \Hyp^3$.
		%\item  
		%
		\indent {\bf EPP:} There exists a proper function $g: \natls \to \natls$ such that for any geodesic subsegment $[a,b]$ of the ray $[o_{\HHH^2}, \xi ) \subset \HHH^2$ lying outside $B_N(o_{\HHH^2})$ (the $N$-ball
		in $\HHH^2$ about $o_{\HHH^2}$),
		the geodesic in $\HHH^3$ joining $\Phi_n (a)$ with $\Phi_n (b)$ lies outside $B_{g(N)}(o_{\HHH^3})$, (the $g(N)$-ball
		about $o_{\HHH^3} \in \HHH^3$). 
	%\end{enumerate}	
\end{prop}

\begin{comment}
\begin{proof}
Apply Proposition \ref{unifcrit1} to the hyperbolic metric spaces $\til K = [1, \xi )$.
\end{proof} 
\end{comment}

\subsection{Models of Ends of Geometric Limits of Surface Groups}
\label{model geometric} We recall some material from \cite{OhD}, where a simplified  analysis of geometric limits of algebraically convergent quasi-Fuchsian groups is given.  This is a special case of a more general result  in \cite{OS}, but will suffice for our purposes.

Let $(\rho_n)$ be a sequence  of quasi-Fuchsian representations of $\pi_1(S)$  converging to $\rho_\infty$ algebraically.
Set $G_n=\rho_n(\pi_1(S))$ and  $G_\infty=\rho_\infty(\pi_1(S))$. Further, (after passing to a subsequence if necessary) 
assume that $(G_n)$ converges geometrically to $\Gamma$.
Let $(\HHH^3/\Gamma)_0$ denote the complement of the $\epsilon_0$-cuspidal part in $\HHH^3/\Gamma$ for a constant $\epsilon_0$ less than the three-dimensional Margulis constant.  We  call $(\HHH^3/\Gamma)_0$ the non-cuspidal part of $\HHH^3/\Gamma$.

In Theorem 4.2 (1) of \cite{OhD} (a special case of a result in \cite{OS}  dealing also with  divergent sequences) it was shown that  there exists a bi-Lipschitz model manifold $\mathbf M_\Gamma^0$ of $(\HHH^3/\Gamma)_0$ admitting an embedding
   into $S\times (0,1)$. 
   %For simplicity, we   regard  $\mathbf M_\Gamma^0$ as a subset of $S\times (0,1)$. 
   Denote by $f_\Gamma : \mathbf M_\Gamma^0 \to (\HHH^3/\Gamma)_0$  the bi-Lipschitz model map.
   As was shown there, this model manifold and the model map can respectively be taken to be the non-cuspidal part of  a geometric limit $\mathbf M_\Gamma$ of Minsky's model manifolds $\mathbf M_n$ of $\HHH^3/G_n$, and the restriction to the non-cuspidal part of the limit of Minsky's model maps $f_n \colon \mathbf M_n \to \HHH^3/G_n$ as $n \rightarrow \infty$. 
  We identify the model manifold $\mathbf M_\Gamma^0$ with its embedding into $S\times (0,1)$, and we regard $f_\Gamma^{-1}$ as an embedding of $(\HHH^3/\Gamma)_0$ into $S \times (0,1)$.
 Since $\mathbf M_\Gamma$ is the union of $\mathbf M_\Gamma^0$ and cusp neighbourhoods, the embedding can be extended to $\mathbf M_\Gamma$.
 
The embedding of the model manifold $\mathbf M_\Gamma^0$ and  the model map $f_\Gamma$ can be taken to have the following properties.
(See Section 4 of \cite{OhD}.)
\begin{enumerate}
\item Each end $e$ of $(\HHH^3/\Gamma)_0$ corresponds under $f_\Gamma^{-1}$ to a level surface $\Sigma \times \{t\}$ for some essential subsurface $\Sigma$ of $S$. 
 More precisely, $\Sigma \times \{t\}$ lies in the frontier of the image of  $f_\Gamma^{-1}$.
\item Every geometrically finite end is sent into $S \times \{0, 1\}$ by $f_\Gamma^{-1}$.
\item
\label{algebraic}
There is an incompressible immersion  $\phi$ of $S$ into $(\HHH^3/\Gamma)_0$ such that the covering of $\HHH^3/\Gamma$ corresponding to $\phi_\ast \pi_1(S)$ coincides with $\HHH^3/G_\infty$. 
\item The image under $f_\Gamma^{-1}$
of the frontier of $(\HHH^3/\Gamma)_0$  consists of disjoint incompressible tori and open annuli built out of horizontal and vertical annuli.
Here we say that an (incompressible) annulus is horizontal when it is embedded in $S \times \{t\}$ for some $t$, and vertical when it has the form $a \times [s,t]$ for some essential simple closed curve $a$.
Each  torus component consists of two horizontal annuli and two vertical ones.
Each open annulus component consists of either one horizontal annulus and two vertical annuli or two horizontal annuli and three vertical annuli.
\end{enumerate}
An embedding $f_\Gamma^{-1}$ and the corresponding model map $f_\Gamma$ satisfying  the above conditions
is said to be {\bf adapted to the product structure}.

A covering associated with  the inclusion $G_\infty \subset \Gamma$ of $\mathbf M_\Gamma$ is homeomorphic to $\HHH^3/G_\infty$, and hence to $S \times (0,1)$.
Its core surface  projects to an immersion of $S$ into $\mathbf M_\Gamma^0$, such that the immersion is homotopic to $f_\Gamma^{-1} \circ \phi$ with $\phi$ as in   Property (\ref{algebraic}) above.
We call such an immersion of $S$ into $\mathbf M_\Gamma^0$ an {\bf algebraic locus}.
An algebraic locus need not be isotopic to a surface of the form $S \times \{t\}$, i.e. it need not be horizontal.
In such a case, an algebraic locus {\bf wraps around} torus boundary components of $\mathbf M_\Gamma^0$.
(See Lemma 4.3 of \cite{OhD} to see that this is the only possibility.)
We sometimes also refer to the immersion $\phi$ in Property (\ref{algebraic}) as an algebraic locus.

If an end $e$ of $(\HHH^3/\Gamma)_0$ corresponds to a level surface $\Sigma \times \{t\}$  for some proper subsurface $\Sigma$ of $S$, then the boundary $\Fr \Sigma$ represents a parabolic element of $\Ga$ contained in a maximal parabolic group  isomorphic to either $\Z$ or $\Z \times \Z$.

\subsubsection{Brick Manifolds}\label{brickmfld}
For later use, we shall give a more precise version of the discussion above in the form of Theorem \ref{thm_osa} below.
A {\bf brick} $B$ is a 3-manifold homeomorphic to $F\times J$, where $F$ is an essential subsurface of $S$  and
 $J$  is either a closed or a half-open interval.
A {\bf brick manifold} is a union of countably many bricks $F_n\times J_n$  glued to each other along essential connected subsurfaces of their horizontal boundaries
$F_n\times \partial J_n$.
We note that {\em the vertical boundary of a brick lies on the boundary of the brick manifold}.
(See Section 4.1 of \cite{OhD} for a more detailed explanation.)

With any
 end of a half-open brick in  a brick manifold $\mathbf M$, we equip either a conformal structure at infinity or an ending lamination. In the first case, the brick is called
 geometrically finite  and in the latter case, it is called simply degenerate. Accordingly, 
each half-open end of a brick is called a geometrically finite or simply degenerate end of $\mathbf M$.
The equipped ending lamination or  conformal structure is called the end invariant.
The union of ideal boundaries corresponding to the geometrically finite ends thus carries a union of conformal structures and is
called the {\bf boundary at infinity} of $\mathbf M$.  Denote the  boundary at infinity by $\partial_\infty \mathbf M$.
A brick manifold equipped with  end invariants is called {\bf a labelled brick manifold}.

A labelled brick manifold is said to admit a block decomposition if the manifold can be decomposed into 
Minsky blocks  \cite{minsky-elc1} and solid tori  such  that 
\begin{enumerate}
	\item Each block has horizontal and vertical directions coinciding with those of bricks.
	\item The block decomposition for a half-open brick agrees with a Minsky model corresponding to its end invariant.
	\item Blocks have standard metrics (as in \cite{minsky-elc1}) and  gluing maps are isometries.
	\item Solid tori are given the structure of  Margulis tubes with coefficients  determined by the block decomposition (as in \cite{minsky-elc1}).
\end{enumerate}

The 
resulting metric on the labelled brick manifold is called
 {\bf a model metric}. The next theorem, which is a combination of Theorem 4.2 and Proposition 4.12 in \cite{OhD} gives the existence of a model manifold corresponding to a geometric limit of Kleinian surface groups.

\begin{theorem}\label{thm_osa} {\rm (Theorem 4.2 and Proposition 4.12 in \cite{OhD}.)}
	Let $S$ be a  hyperbolic surface of finite area.
	Let $(\rho_n : \pi_1(S) \to \pslc)$ be a sequence of
	weakly type-preserving representations, converging geometrically 
	to $\Gamma$. Set $N = \HHH^3 / \Gamma$, and let $N_0$ denote the non-cuspidal part of $N$.
	Then there exists a labelled brick manifold $\mathbf M^0_\Gamma$ admitting  a block decomposition    and  a $K$-bi-Lipschitz homeomorphism to $N_0$
	such that the following hold: 
	\begin{enumerate}
		\item The constant $K$ depends only on $\chi(S)$.
		\item
		Each component of $\partial \mathbf M^0_\Gamma$ is either a torus or an open annulus.
		\item $\mathbf M^0_\Gamma$ has only countably many ends, and no two distinct ends lie on the same level surface $\Sigma \times \{t\}$.
		\item
		There is no properly embedded incompressible annulus in $\mathbf M^0_\Gamma$ whose boundary components lie on distinct boundary components.
		\item
		If there is an embedded, incompressible half-open annulus $S^1 \times [0,\infty)$ in $\mathbf M^0_\Gamma$ such that $S^1 \times \{t\}$ tends to a wild end $e$ of $\mathbf M^0_\Gamma$ as $t \rightarrow \infty$ (see Remark \ref{wild} below), then its core curve is freely homotopic into an open annulus component of $\partial \mathbf M^0_\Gamma$ tending to $e$.
		\item The manifold $\mathbf M^0_\Gamma$
		 is (not necessarily properly) embedded in $S \times (0,1)$ in such a way that  each brick has the form $F \times J$ where $F$ is an essential subsurface  of $S$ and $J \subset (0,1)$ is an interval. Also the product structure of  $F \times J$ is compatible with that of $S \times (0,1)$. The ends of  geometrically finite bricks lie in $S \times \{0,1\}$. Further, for a brick $F \times J$ equipped with  
		a (topological) product structure, the vertical boundary $(\partial F) \times J$ is necessarily
	contained in the boundary $\partial \mathbf M^0_\Gamma$ of $\mathbf M^0_\Gamma$. 
	\end{enumerate}
\end{theorem}

The labelled brick manifold $\mathbf M^0_\Gamma$ of Theorem \ref{thm_osa}
is called a {\bf model manifold} for $N_0$ -- the
non-cuspidal part of the geometric limit. 
As was explained in the previous section, the model manifold of $\mathbf M_\Gamma$ is obtained as a geometric limit of the model manifolds $\mathbf M_n$ of $\HHH^3/\rho_n(\pi_1(S))$.
By removing cusp neighbourhoods from $\mathbf M_\Gamma$, we get $\mathbf M^0_\Gamma$.

\begin{rmk}
\label{wild}
	In general, when $\pi_1(N)$ is infinitely generated,
	the non-cuspidal part of a geometric limit $N$ as in Theorem \ref{thm_osa}   may contain an end all of whose open
	neighbourhoods contain infinitely many distinct relative ends.
	We call such an end {\bf wild}. In this case we have a sequence of relative ends accumulating (under the model map $f_\Gamma$) to some $\Sigma \times \{ t \}$, where $\Sigma$ is an essential subsurface  of $S$.
\end{rmk}

\begin{rmk}
\label{osrmk}
In \cite{OS}, the authors further show that given a family of end-invariants on a 
labelled brick manifold satisfying the conclusions of Theorem \ref{thm_osa} above,  there  exists a model manifold with those end-invariants provided only that there are no two homotopic parabolic curves or two homotopic ending laminations. Further, such a manifold is unique up to bi-Lipschitz homeomorphism.
\end{rmk}

\subsection{Special Conditions on Ends}
\label{ends}
 Recall that we have fixed a Fuchsian group $G$ with
limit set $\Lambda_G$  homeomorphic to $S^1$. Equivalently, if $S$ is closed then $\partial_\infty G = S^1$ and if $S$ is non-compact, then the relative hyperbolic boundary $\partial_h G = S^1$.
Suppose that a sequence of quasi-Fuchsian groups $(G_n = \rho_n(G))$ converges geometrically to the geometric limit $\Gamma$.
We denote by $\Lambda_{G_n}$  the limit set of $G_n$, and
by $c_n : S^1=\Lambda_G \rightarrow \Lambda_{G_n}$  the corresponding Cannon-Thurston map. We assume that $(\rho_n)$ converges to $\rho_\infty$ algebraically, and set $G_\infty=\rho_\infty(G)$.
Recall that we have a model manifold $\mathbf M_\Gamma$ with a model map $f_\Gamma \colon \mathbf M_\Gamma \to \HHH^3/\Gamma$ which are geometric limits of the model manifolds $\mathbf M_n$ of $\HHH^3/G_n$ and $f_n \colon \mathbf M_n \to \HHH^3/G_n$.
 We regard its non-cuspidal part  $\mathbf M_\Gamma^0$ as being embedded in $S \times (0,1)$.
Since  $(\mathbf M_n)$ converges to $\mathbf M_\Gamma$ geometrically, there exist $K_n$-bi-Lipschitz homeomorphisms $\mathbf h_n\colon B_{R_n}(\mathbf M_n, \mathbf x_n) \rightarrow B_{K_n R_n}(\mathbf M_\Gamma, \mathbf x_\infty)$, where $R_n \rightarrow \infty$ and $K_n \to 1$. We can also assume that $\mathbf x_\infty$ lies in the algebraic locus (i.e.\ an immersion of $S$ into the geometric limit whose fundamental group corresponds to the algebraic limit, see \S \ref{model geometric}).
In the same way, corresponding to the geometric convergence of $(G_n)$ to $\Gamma$, there exist $K_n$-bi-Lipschitz homeomorphisms $h_n \colon B_{R_n}(\HHH^3/G_n, x_n) \to B_{K_nR_n}(\HHH^3/\Gamma, x_\infty)$ with $x_\infty$ lying on the algebraic locus.

We shall now describe the conditions that appear in the main theorem.\\

\noindent{\bf Coupled ends.}
%We shall consider ends of $(\HHH^3/\Gamma)_0$.
Let $e$ be a simply degenerate end  of $(\HHH^3/\Gamma)_0$. Then there is a neighbourhood $E$ of $f_\Gamma^{-1}(e)$ of the form $\Sigma \times (t_1, t_2)$ where  $\Sigma$ is an essential subsurface of $S$, and either $\Sigma \times \{t_1\}$ or $\Sigma \times \{t_2\}$ corresponds to $e$.

Alternatively, let 
 $e$ be wild. Then there is a neighbourhood $E$   homeomorphic to the complement of countably many pairwise disjoint neighbourhoods of (simply degenerate or wild) ends $\Sigma_k \times [s_k, s'_k]$ in $\Sigma \times (t_1, t_2)$, where $e$ corresponds to either $\Sigma \times \{t_1\}$ or $\Sigma \times \{t_2\}$. Here $\Sigma_k$ is an essential subsurface of $\Sigma$ and $t_1 < s_k < s'_k < t_2$. If $e$ corresponds to $\Sigma \times \{t_1\}$, then both $(s_k)$ and $(s_k')$ accumulates to $t_1$ from below; similarly if $\Sigma \times \{t_2\}$ corresponds to $e$, they accumulate from above. 
\begin{definition}
\label{def coupled}
	Suppose that an end $e$ of $(\HHH^3/\Gamma)_0$ corresponds (under $f_\Gamma^{-1}$)  to $\Sigma \times \{t\}$ in $\mathbf M_\Gamma^0$ for some $t \in (0,1)$ and an essential subsurface $\Sigma$ of $S$.
	\begin{enumerate}
		\item We say that an end of $(\hyperbolic^3/\Gamma)_0$ is {\bf algebraic} if it has a neighbourhood which is a homeomorphic image of a neighbourhood of an end of $(\hyperbolic^3/G_\infty)_0$ under the covering projection $q\colon \hyperbolic^3/G_\infty \to \HHH^3/\Gamma$ associated with the inclusion of $G_\infty$ into $\Gamma$.
		\item	
		We call the end $e$ {\bf upward} if $f_\Gamma^{-1}(e)$ corresponds to $\Sigma \times \{t\}$ and $\Sigma \times (t-\epsilon, t)$ intersects $\mathbf M_\Gamma^0$ for every small $\epsilon > 0$, else it is called   {\bf downward}.
		\item
		When $e$ is simply degenerate and upward, we say that it is {\bf coupled} if there is a downward end $e'$ of $(\HHH^3/\Gamma)_0$ such that the following hold if we choose an embedding of $\mathbf M_\Gamma^0$ into $S \times (0,1)$ appropriately.		
		%are downward  ends $e'_1, \dots, e_d'$ of $(\HHH^3/\Gamma)_0$ which are either simply degenerate or wild, such that the following hold if we choose an embedding of $\mathbf M_\Gamma^0$ into $S \times (0,1)$ appropriately.		
		\begin{enumerate}
		\item $f_\Gamma^{-1}(e')$ corresponds to $\Sigma' \times \{t'\}$ with $t' >t$ and an essential subsurface $\Sigma'$ of $S$.
		\item There is a boundary component $A$ of $\mathbf M_\Gamma^0$ which abuts both on $f_\Gamma^{-1}(e)$ and $f_\Gamma^{-1}(e')$.
		(We defined the term 'abutting' for $(\HHH^3/\Gamma)_0$, but abuse the term also for the model manifold $\mathbf M_\Gamma^0$.)
		\item There is an essential subsurface $\Sigma'_+$ of $\Sigma' \times \{t'+\epsilon'\}$ which intersects $A$ at its boundary,  such that   the surface $\mathbf h_n^{-1}(\Sigma'_+)$  is \lq parallel into'  $\mathbf h_n^{-1}(\Sigma \times \{t-\epsilon\})$  in $\mathbf M_n$ for sufficiently large $n$, \ie  by moving $\mathbf{h}_n^{-1}(\Sigma_+')$ vertically, it can be isotoped into $\mathbf{h}_n^{-1}(\Sigma \times \{t-\epsilon\})$. 
		\end{enumerate}
%		$\Image f \cap ((\Int \Sigma) \times [t,t'] )= \emptyset$ and there is a boundary component of $(\HHH^3/\Gamma)_0$ which is an open annulus one of whose ends tends to $e$ whereas the other tends to $e'$.
	Similarly when $e$ is downward, we call it coupled if there is an upward ends $e'$  satisfying analogous conditions to the upward case.
	% such that $\Image f \cap ((\Int \Sigma) \times [t,t'] ) = \emptyset$ and there is a boundary component of $(\HHH^3/\Gamma)_0$ which is an open annulus one of whose ends tends to $e$ whereas the other tends to $e'$.
		%The $\integers$-cusp corresponding to such an open annulus is called {\bf conjoining}.
		\item $e'$ above is called a {\bf partner} of $e$.
		\item We say that a simply degenerate end of the (non-cuspidal) algebraic limit $(\HHH^3/G_\infty)_0$ is coupled when the corresponding algebraic end of $(\HHH^3/\Gamma)_0$ is coupled.
	\end{enumerate}
\end{definition}
We note that we can change the embedding of $\mathbf M_\Gamma^0$ into $S \times (0,1)$ preserving the algebraic locus to another one adapted to the product structure without changing the combinatorial structure of the brick decomposition so that 
%in (3), all $t_k'$  are equal to some $t' \in (0,1)$, 
in condition (3) above, the surface $\Sigma'_+$ may be taken to be a subsurface of $\Sigma$.
%, and  $\Sigma \times [t,t']$ lies outside $\mathbf M_\Gamma^0$.
We also note that a coupled end may have more than one partner.

\begin{comment}
An end of $(\HHH^3/\Gamma)_0$ is said to be {\bf coupled} if it is sent by $f_\Gamma$ to $\Sigma \times \{t\}$ in $S\times (0,1)$ for an essential subsurface $\Sigma$ of $S$ and there is another end $e'$ of $(\HHH^3/\Gamma)_0$ sent to $\Sigma \times \{t'\}$ such that $\Sigma \times [t,t']$ is contained in $S \times (0,1) \setminus f_\Gamma((\HHH^3/\Gamma)_0)$.
We call $e'$ as above the counterpart of the coupled end $e$.
\end{comment}

Now, suppose that $e$ is algebraic, and let $\tilde e$ be a coupled end of $(\HHH^3/G_\infty)_0$ which is projected down to $e$ by $q$.
Let $A$ be an open annulus boundary component of $(\HHH^3/\Gamma)_0$ such that one of its ends abuts on $e$ whereas the other end abuts on its partner $e'$.
The $\integers$-cusp corresponding to such an open annulus, as also the annulus itself, are called {\bf conjoining}.
A conjoining $\integers$-cusp  lifts, in $(\HHH^3/G_\infty)_0$,
 to a $\integers$-cusp $U$ abutting on $\tilde e$. The cusp $U$ corresponds to an annular neighbourhood $A$ of a parabolic curve on $S$, and $A$ has two sides on the surface $S$.
Let $e_0$ be an end of $(\HHH^3/G_\infty)_0$ which lies on the other side of $U$. (A priori $e_0$ may coincide with $\tilde e$ itself if $A$ is non-separating).
If $e_0$ is  geometrically infinite (i.e.\ simply degenerate), then, by the covering theorem \cite{thurstonnotes, canary-cover, OhQ}, this end also has a neighbourhood  embedded homeomorphically in $(\HHH^3/\Gamma)_0$ under the covering projection $q$.
In particular the projection $q(\Fr U)$ of the open annulus $\Fr U$ must abut on an algebraic end.
This contradicts the assumption that  $q(\Fr U)$ abuts on $e'$, which cannot be algebraic. 
Hence, $e_0$ must be geometrically finite; in particular, $e_0$ is distinct from $\tilde e$.
We call a $\integers$-cusp of $(\HHH^3/G_\infty)_0$ separating  a geometrically finite end from a geometrically infinite end {\bf finite-separating}.
What we have just shown can be stated as follows.
\begin{lemma}
Any conjoining cusp in $(\HHH^3/\Gamma)_0$ lifts to a finite-separating cusp in $(\HHH^3/G_\infty)_0$.
\end{lemma}

\noindent {\bf Twisted and Untwisted cusps.}
The $\integers$-cusps of $\HHH^3/\Gamma$ not corresponding to cusps of $S$ are classified into two types: twisted and untwisted. We shall describe these now.
Let $P$ be a $\integers$-cusp neighbourhood of $\HHH^3/\Gamma$ corresponding to a maximal parabolic group generated by  $\gamma \in \Gamma$, and  not coming from a cusp of $S$.
 %a Fuchsian group $G$ corresponding to the hyperbolic structure.

Then, there is a sequence of loxodromic elements $(\gamma_n \in G_n)$ converging to $\gamma$ by the definition of geometric limits.
Let $U_n$ be a Margulis tube of $\HHH^3/G_n$ whose core curve is represented by $\gamma_n$.
We can choose  $U_n$ so that they converge to $P$ geometrically as  $\HHH^3/G_n$ converges geometrically to $\HHH^3/\Gamma$.
Take an  annular core $A$ of $\partial P$, and pull it back
to an annulus $A_n$ on $\partial U_n$ by $h_n^{-1}$. 
Now, consider a meridian
$m_n$  on $\partial U_n$ (\ie $m_n$ is an essential simple closed curve on $\partial U_n$ bounding a disc in $U_n$)   and a longitude $l_n$, \ie a core curve of $A_n$  generating  $\pi_1(\partial U_n)$.
Let $s_n$ be a simple closed curve  whose length  with respect to  the induced metric on $\partial U_n$ is shortest among the simple closed curves intersecting $l_n$ at one point.
Then in $\pi_1(U_n)$, we can express $[s_n]$ as $[m_n]+ \alpha_n [l_n]$ for some $\alpha_n \in \integers$.
If $\alpha_n$ is bounded as $n \rightarrow \infty$, we say that the cusp $P$ is  {\bf untwisted};  else it is said to be {\bf twisted}.

A description of twisted and untwisted cusps may also be given  using the hierarchy machinery of  Masur-Minsky \cite{masur-minsky2}.
Let $H_n$ be a hierarchy of tight geodesics in the curve complex of $S$ corresponding to the quasi-Fuchsian group $G_n$ having $\rho_n$ as a marking.
Then the cusp $P$ is twisted  if and only if $H_n$ contains a geodesic $g_n$ supported on an annulus whose core curve is freely homotopic to $\gamma_n$ and whose length goes to $\infty$ as $n \rightarrow \infty$.\\

%\noindent {\bf Finite-separating cusps}\\
%By  tameness  \cite{bonahon-bouts}, for any geometrically infinite end $e$ of $(\HHH^3/G_\infty)_0$ there is a neighbourhood $U_e$ of $e$  homeomorphic to $\Sigma \times \reals$ for an incompressible subsurface $\Sigma$ of $S$.
%Each frontier component of $\Sigma \times \{s\}$ lies in the neighbourhood of a cusp.
%We call each  such cusp a cusp {\bf attached} to $e$.
%We note that each cusp is attached to one or two ends.
%
%If a cusp is attached to both a geometrically finite end and a geometrically infinite end, we say that it is {\bf finite-separating}.
%We also note that every geometrically infinite end of  $(\HHH^3/G_\infty)_0$ is mapped to a geometrically infinite end of $(\HHH/\Gamma)_0$ by Thurston's covering theorem \cite{thurstonnotes}
%(see also Ohshika \cite{OhQ} and Canary \cite{canary}).\\

\noindent {\bf Crown domains and crown-tips.}
%Suppose that $(\HHH^3/G_\infty)_0$ has a coupled simply degenerate end $e$ which has  well-approximated ending lamination and an untwisted conjoining cusp $P$ abutting on it. Without loss of generality, we can assume that $e$ is an upward  end and let $e'$ be the partner of the end $\bar e$ which is the projection of $e$ by $q\colon (\HHH^3/G_\infty)_0 \to (\HHH^3/\Gamma)_0$.
Let $e$ be a simply degenerate end of $(\HHH^3/G_\infty)_0$ with ending lamination $\lambda$.
Let $\Sigma$ be the minimal supporting surface of $\lambda$,  i.e.\ an essential subsurface of $S$ containing $\lambda$ and minimal with respect to  inclusion (up to isotopy).
Let $\sigma$ be a component of $\Fr \Sigma$.
Fixing a hyperbolic metric on $S$, we can assume that both $\sigma$ and $\lambda$ are geodesic on $S$.
We consider  their pre-images $\tilde \sigma$ and $\tilde \lambda$ in $\HHH^2$.
%Since $P$ abuts on $e$, a component $c_0$ of $\tilde  \sigma$ and some boundary leaves of $\tilde \lambda$ bound between them a crown domain $C$, which is an ideal polygon with countably many vertices.
A crown domain $C$ for $(\lambda, \sigma)$ is an ideal polygon in $\HHH^2 \setminus (\tilde \sigma \cup \tilde \lambda)$ with countably many vertices bounded by a component $\sigma_0$ of $\tilde \sigma$ and countably many leaves of $\tilde \lambda$.
A vertex of $C$ which is not an endpoint of $\sigma_0$ is called  {\bf a tip} of the crown domain $C$ or simply a {\bf crown-tip} (for $(\lambda, \sigma)$).\\

\noindent {\bf Well-approximated crown domains in ending laminations.}
To state a sufficient condition for pointwise convergence, we need to introduce a subtle condition concerning  geometric convergence to coupled geometrically infinite ends as follows.
Let $e$ be a coupled simply degenerate end of $(\HHH^3/G_\infty)_0$.
We  assume that $e$ is an upward end.
As mentioned earlier, $e$ has a neighbourhood  projecting homeomorphically to a neighbourhood of a simply degenerate end $\bar e$ of $(\HHH^3/\Gamma)_0$.
We can define the same property when $e$ is a lower end by turning everything upside down.

Recall that we have an embedding $f_\Gamma^{-1} \colon (\HHH^3/\Gamma)_0 \rightarrow S \times (0,1)$ adapted to the product structure as described above, and the end $f_\Gamma^{-1}(\bar e)$ corresponds to a level surface $\Sigma \times \{t\}$ for some essential  subsurface $\Sigma$ of $S$.
%From now on, we identify $(\HHH^3/\Gamma)_0$ and its image under $f_\Gamma$. It will be convenient to assume that $\Sigma \times \{t\}$ lies on the frontier of $f_\Gamma(\bar e)$ rather than its interior (in particular $\Sigma \times \{t\}$ is {\it not} contained in the image of $f_\Gamma$).
%
Since $e$ is assumed to be  coupled,  there exists $t' > t$ such that $\Sigma' \times \{t'\}$ corresponds to a downward end $e'$ of $(\hyperbolic^3/\Gamma)_0$ which is either simply degenerate or wild.
There is  at least one conjoining annulus $A$ abutting on  $\bar e$ as well as $e'$.

%Without loss of generality, we can choose $t=t'$.
%Now,  for sufficiently large $n$, there is an approximate isometry $h_n$ between $\HHH^3/G_n$ and $\HHH^3/\Gamma$.
%この部分を変える．
Now, as in the definition of a coupled end, pick a surface $\Sigma_-=\Sigma \times \{t-\epsilon\}$ in $\mathbf M_\Gamma^0$ for some small $\epsilon>0$. 
Take a surface $\Sigma'_+$ lying on $\Sigma' \times \{t'+\epsilon'\}$ as in the condition (3)-(c) in Definition \ref{def coupled}, which is assumed to intersect $A$ at its boundary.
%Also assume that  $\Sigma_-$ is  contained in the image of $h_n$ for sufficiently large $n.$ 
%Define $\Sigma'_+$ to be $\cup_{k=1}^d \Sigma'_k \times \{t_+\} \cup \cup_{l=1}^e A'_l$ with $t_+ >t'$.
%We have an incompressible subsurface $\Sigma''_+$ of  $\Sigma_+'$ on which $A$ abuts such that $\mathbf h_n^{-1}(\Sigma''_+)$ is parallel into of $\mathbf h_n^{-1}(\Sigma_-)$.
Let us denote by $k_n$ an embedding of   $\mathbf h_n^{-1}(\Sigma'_+)$ into $\mathbf h_n^{-1}(\Sigma_-)$ realising the parallelism given in the condition (3)-(c).
Then,  $\mathbf h_n \circ k_n \circ \mathbf h_n^{-1}|\Sigma'_+$ gives an embedding of $\Sigma'_+$ into $\Sigma_-$. Denote this embedding by $\Psi_n$.  
Fix a complete hyperbolic structure on $\Sigma_-$ making each component of $\Fr \Sigma_-$ a cusp, and 
isotope $k_n$ so that each frontier component of $\Psi_n(\Sigma'_+)$ that is not contained in $\Fr 
\Sigma_-$ is a closed geodesic in $\Sigma_-$.
%Identifying $\Sigma_\pm$ with $\Sigma$ via parallelism in $S \times (0,1)$, we obtain an automorphism $\Psi_n$ of $\Sigma$.

Let $\lambda$ be the ending lamination of $e$.
Then $\Sigma_-$ is regarded as the minimal supporting surface of $\lambda$.
Let $\sigma$ be a frontier component of $\Sigma_-$ in $S$.
Then $\rho_\infty(\sigma)$ represents a parabolic curve.
Let $P$ be a cusp in $\HHH^3/\Gamma$ corresponding to $\rho_\infty(\sigma)$, where $G_\infty$ is regarded as a subgroup of $\Gamma$.
Suppose that the annulus $A$ which abuts on both $e$ and $e'$ as above is the boundary of $P$.
Realise $\sigma$ and $\lambda$ as geodesics in $S$  (with respect to a fixed hyperbolic metric) and consider their lifts $\tilde \sigma$ and $\tilde \lambda$ to $\HHH^2$.
Let $C$ be a crown domain of $(\lambda, \sigma)$.
Its projection $p(C)$ into $S$ is an  annulus having finitely many frontier components one of which is the closed geodesic $\sigma$ and the others are bi-infinite geodesics.
Let $\lambda_C$ denote the union of the boundary leaves of $p(C)$ other than $\sigma$. 
We isotope $k_n$ fixing the boundary of $\Sigma'_+$ so that $\Psi^{-1}_n(\lambda_C)$ is geodesic.

We say that the crown domain $C$ is {\bf well approximated} for $(\rho_n)$ if  the closure of the union of the geodesics homotopic to $(\Psi_n^{-1}(\lambda_C))$ converges in the Hausdorff topology to a geodesic lamination which can be realised on a pleated subsurface in a neighbourhood of a partner $e'$ of $e$ conjoined by
 the boundary $A$ of $P$ as above. 
The pleated subsurface in question is thought of as a map from 
%an incompressible subsurface $\Sigma^0_+$ of $\Sigma'_+$ containing 
$\Sigma'$ homotopic to the restriction of the model map $f_\Gamma$ to $\Sigma' \times \{t'\}$.

%{\bf not} representing the ending lamination of $e'$. Here a choice of the surface $\Sigma_+$ is implicit.
If $e'$ is simply degenerate, this  condition is equivalent to saying that $(\Psi_n^{-1}(\lambda_C))$ does not converge to leaves of the ending lamination of $e'$.
In particular, in the simply degenerate case, the choice of $\Sigma'_+$ or $\epsilon'$ is irrelevant   since  all such surfaces are part of $\Sigma' \times \{t'+\epsilon'\}$ which  are parallel to each other in a neighbourhood of $e'$. The choice  {\em is} however relevant if $e'$ is wild. 
In this case, we say that $C$ is well approximated if we can choose $\Sigma'_+$ or $\epsilon'$ so that the above condition  holds.
%Note further that when $e'$ is simply degenerate,  this is equivalent to saying that $(\Psi_n(\lambda_C))$ converges to a projective lamination which does not represent the ending lamination of $e'$.\\
\begin{rmk}
\label{leaf of EL}
In the above definition, we assumed that the geodesic lamination which is a limit of $(\Psi_n^{-1}(\lambda_C))$ is realisable.
In practice, it suffices to assume that  for at least one leaf $\ell$ of $\lambda_C$, the limit geodesic  of $(\Psi_n^{-1}(\ell))$ is realisable.
Indeed, if a leaf $\ell_0$ of the limit lamination of $(\Psi_n^{-1}(\lambda_C))$ is not realisable, it is a leaf of an ending lamination of some simply degenerate end of $(\HHH^3/\Gamma)_0$.
If $\ell_1$ is the limit geodesic of $(\Psi_n^{-1}(\ell'))$ where $\ell'$ is a geodesic in $\lambda_C$ adjacent to $\ell_0$, then $\ell_1$ is asymptotic to $\ell_0$. Hence it is also a leaf of the ending lamination.
Inductively, no geodesics in the limit of $(\Psi_n^{-1}(\lambda_C))$ are realisable.
\end{rmk}

\begin{rmk}
\label{Brock's case}
In the study of Brock's examples \cite{brock-itn} carried out in \cite{mahan-series2}, $e$ has only one partner $e'$, which is simply degenerate, and $\Sigma'$ is homeomorphic to $\Sigma$.
Further, the natural
 embedding of $\mathbf M_\Gamma$ into $S \times (0,1)$
 ensures that the ending lamination of $e'$ is not homotopic to the ending lamination $\lambda$ of $e$ in $S \times (0,1)$. 
We can further take  $\Psi_n$ to be a pseudo-Anosov map on $\Sigma$ which fixes $\lambda$ if we identify $\Sigma$ and $\Sigma'$ by a parallelism in $S \times (0,1)$.
For a crown domain $C$ of $(\lambda, \sigma)$, any leaf of $\lambda_C$ is a leaf of $\lambda$ and is therefore   dense in the latter.
Since $\lambda$ can be realised by a pleated surface homotopic to the inclusion map into a neighbourhood of $e'$, the crown domain $C$ is well approximated.
Thus Brock's examples satisfy the well-approximation condition.
\end{rmk}

\begin{rmk}
\label{partial realisation}
In the definition of well-approximation, we considered the union of boundary leaves of a crown domain $\lambda_C$ rather than the entire ending lamination $\lambda$.
Therefore,  the homeomorphism $\Psi_n^{-1}$ need not be defined on the entire $\Sigma$, and $\Sigma'$ can be a proper subsurface of $\Sigma$.
%In general, the condition that the limits of the geodesics in  $(\Psi_n^{-1}(\lambda_C))$ are realisable is weaker than the condition that the Hausdorff limit of $\Psi_n^{-1}(\lambda)$ is realisable, even if $\Phi_n$ can be defined on the entire $\Sigma$.
%For instance, it may be possible that the sequence $(\Psi_n^{-1}(\lambda_C))$ converges to geodesics both  whose ends spiral around a frontier component $c$ of $\Sigma'_k$ and $\Sigma'_k$ lies in the interior of $\Sigma'$.
%In such a case, the Hausdorff limit of the entire $\Psi_n^{-1}(\lambda)$ contains $c$ as a leaf, but $c$ is not realisable.
\end{rmk}

%We say that an end of $(\HHH^3/G_\infty)_0$ is coupled if
%(under the natural inclusion of ends) its image in  $(\HHH^3/\Gamma)_0$ is coupled. Similarly for twisted annuli in $(\HHH^3/G_\infty)_0$
%and ends of $(\HHH^3/G_\infty)_0$ having well-approximated ending laminations.

%We have thus defined well-approximated ending laminations, whether the counterpart of the coupled ends is simply degenerate or wild.
%We shall next see that in fact, the counterpart must be simply degenerate if the ending lamination is well approximated.
%
%\begin{lemma}
%\label{no wild counterpart}
%If the ending lamination of $e$ as above has well-approximated ending lamination, then $e'$ cannot be wild.
%\end{lemma}
%\begin{proof}
%We consider a surface $\Sigma_-$ as we took in the description above.
%We say that a downward simply degenerate end, corresponding to $F \times \{s\}$ with $s< t_+$, or a torus end, corresponding to $\partial (A \times [s', s])$ with $s< t_+$, is said to be {\em immediately beneath} $e'$ if $F$ (resp. $A$) is a subsurface of $\Sigma$ and $F\times (s,t)$ (resp. $A \times (s,t)$) is contained in $f_\Gamma((\HHH^3/\Gamma)_0)$.
%
%\end{proof}

\subsection{Statements and scheme} With the background and terminology above, we can restate Theorems \ref{introthm1}
and \ref{introthm2} more precisely.

\begin{theorem}
\label{pointwise convergence}
Let $S=\HHH^2/G$ be a hyperbolic surface of finite area, and let $(\rho_n: \pi_1(S) \rightarrow \pslc)$ be a sequence of quasi-Fuchsian groups (obtained as quasi-conformal deformations of $G$) converging algebraically to $\rho_\infty: \pi_1(S) \rightarrow \pslc$.
We set $G_n=\rho_n(\pi_1(S))$ and $G_\infty=\rho_\infty(\pi_1(S))$.
Suppose that $(G_n)$ converges geometrically to a Kleinian group $\Gamma$.
Then the Cannon-Thurston maps $c_n : S^1 (=\Lambda_G) \rightarrow \Lambda_{\rho_n(\pi_1(S))}$ for  $\rho_n$ do not converge pointwise to 
the Cannon-Thurston map  $c_\infty : S^1 \rightarrow \Lambda_{\rho_\infty(\pi_1(S))}$ for $\rho_\infty$ if and only if all of the following conditions hold:
\begin{enumerate}
\item there is a coupled
simply degenerate end $e$ of $(\HHH^3/G_\infty)_0$ with ending lamination $\lambda$,
\item there is an untwisted conjoining cusp $U$ abutting on the projection of $e$ to $(\HHH^3/\Gamma)_0$, such that $U$ corresponds to a parabolic curve $\sigma$,
\item there is a  well-approximated crown domain for $(\lambda, \sigma)$.
\end{enumerate}
\end{theorem}

\begin{theorem}
\label{non-continuous}
In the setting of Theorem \ref{pointwise convergence}
above, suppose that $(\HHH^3/G_\infty)_0$ has a coupled simply degenerate end with an untwisted conjoining cusp abutting on it such that the corresponding crown domain is well approximated.
Let $\lambda_+$ be the union of the upper parabolic curves and the upper ending laminations, and $\lambda_-$ the union of the lower parabolic curves and the lower ending laminations for $\HHH^3/G_\infty$.
Then for $\zeta \in \Lambda_G=S^1$, the sequence $(c_n(\zeta))$ does not converge to $c_\infty(\zeta)$ if and only if $\zeta$ is a tip of a crown domain $C$ for $(\mu,  \sigma)$ where 
\begin{enumerate}
\item
$\mu \cup  \sigma$ is contained  in either $\lambda_-$ or $\lambda_+$;
\item
$\sigma$  corresponds to an untwisted conjoining cusp abutting on the end for which $\mu$ is the ending lamination; and
\item 
$C$ is  well approximated.
\end{enumerate}
\end{theorem}
A word of clarification here. The algebraic limit may contain both upper and lower ending laminations corresponding to {\it different} subsurfaces. Each of these is a potential source of discontinuity {\it provided} they satisfy the second and third conditions. 
%The case where  $\mu \cup  \sigma$ is contained in both an upper and a lower ending lamination in the algebraic limit corresponds
%precisely  to the doubly degenerate case dealt with in \cite{mahan-series2}.
We now briefly describe the scheme  we shall follow to  prove the above two theorems:

\begin{enumerate}
	\item In Section \ref{necessity}, we shall show that  if there is a coupled  geometrically infinite end  $(\HHH^3/G_\infty)_0$  with an untwisted conjoining cusp abutting on it, and the corresponding crown domain is well approximated, then at the corresponding crown-tips, the sequence of Cannon-Thurston maps {\em do not} converge.
	\item In Section \ref{noncrown}, we shall show that at points other than crown-tips, the sequence of Cannon-Thurston maps {\em always} converge pointwise.
	\item Finally, in Section \ref{crowntips}, we shall prove the remaining assertion: if a crown $C$ 
	\begin{itemize}
		\item is either not well approximated
		\item  or does not come from the ending lamination of a simply degenerate end $e$ of   $(\HHH^3/G_\infty)_0$ and a parabolic curve corresponding to an untwisted conjoining cusp abutting on $e$,
	\end{itemize} 
 then at the tips of $C$ the sequence of Cannon-Thurston maps {\em do} converge pointwise.
\end{enumerate}

\section{Necessity of conditions}\label{necessity}
In this section, we shall prove  the \lq only if' part of Theorem \ref{pointwise convergence}, and  the \lq if' part of Theorem \ref{non-continuous}.
As in the  definition of well-approximated crown domains in Section \ref{ends}, we assume that $(\HHH^3/G_\infty)_0$ has  a coupled simply degenerate end $e$ with   an untwisted conjoining cusp $P$ abutting on its projection $\bar e$ in $(\HHH^3/\Gamma)_0$.
Let $\sigma$ be a parabolic curve representing $P$, and $\lambda$ the ending lamination of $e$, both of which we realise as geodesics in $S$.
We lift them to $\HHH^2$. Consider a crown domain $C$ and let $\zeta$ be a tip of $C$.
We assume that $C$ is well approximated.
We shall show that $(c_n(\zeta))$ does not converge to $c_\infty(\zeta)$. This will  prove  both the \lq only if' part of Theorem \ref{pointwise convergence}, and  the \lq if' part of Theorem \ref{non-continuous} at the same time.

The proof is similar to that of discontinuity for Brock's example   dealt with in \cite{mahan-series2}.
 Let $\Sigma$ denote the minimal supporting surface of $\lambda$, i.e.\ an essential subsurface of $S$ containing $\lambda$ and minimal up to isotopy with respect to  inclusion.
 Since $\sigma$ is finite-separating, there exists an essential subsurface $B$ of $S$ such that
 \begin{enumerate}
 	\item $\Sigma \cap B = \sigma$ 
 	\item $B$ corresponds to an upper  geometrically finite end in the algebraic limit $(\HHH^3/G_\infty)_0$
 \end{enumerate}
 By assumption, the crown domain $C$ is well approximated.
 Therefore, there are essential subsurfaces $\Sigma_-$ and $\Sigma_+$  contained in neighbourhoods of $e$ and its partner $e'$ respectively with the following properties.
% These correspond to $\Sigma'_+$ and its parallel image in $\Sigma_-$ in Definition \ref{def coupled}.
(Here instead of the model manifold $\mathbf M_\Gamma$ used in \S \ref{ends}, we use $\HHH^3/\Gamma$ itself.)
The surface  $h_n^{-1}(\Sigma_+')$  is parallel into $h_n^{-1}(\Sigma_-)$ in $\HHH^3/G_n$. 
Let $\Psi_n$ denote  an embedding from $\Sigma_+'$ to $\Sigma_-$ induced by this parallelism in $\HHH^3/G_n$  via $h_n$ as in Section \ref{ends}.
The assumption of well-approximated crown domain says that for any geodesic side $\lambda_C$ ($\neq \sigma$) of $p(C)$, the geodesic $(\Psi_n^{-1}(\lambda_C))$ converges in the Hausdorff topology (with respect to a fixed hyperbolic metric) to a geodesic lamination $\mu$ which can be realised by a pleated surface homotopic to the inclusion of $\Sigma_+'$.
Therefore, the realisation of $\lambda_C$ by a pleated surface from $S$ to $\HHH^3/G_n$ inducing $\rho_n$ between  fundamental groups can be further pushed forward by $h_n$ to a quasi-geodesic realisation. The sequence of  quasi-geodesic realisations thus obtained converges to a family of geodesics realised in $\HHH^3/\Gamma$.  We denote the latter by $\mu^\ast$.

Now let $(\kappa_n)$ be a sequence of  bi-infinite geodesics in $\HHH^3/G_n$ asymptotic in one direction to the closed geodesic representing $\rho_n(\sigma)$ and in the other to (an end-point of a leaf of) the realisation
of $\lambda_C$.  
%Then $\kappa_n$ converges to a bi-infinite geodesic $\kappa_G$ in
Lift $\kappa_n$ to a geodesic $\tilde \kappa_n$ in $\HHH^3$ asymptotic to a lift of a leaf of $\lambda_C$. Also assume that  $\tilde \kappa_n$ contains a basepoint $o_n^\mu$ lying within a bounded distance of $o_{\HHH^3}$.
Since  $(\Psi_n^{-1}(\lambda_C))$ converges to  $\mu^\ast$,   $(\tilde\kappa_n)$ converges  to a geodesic with  distinct endpoints. One of these, $p_\sigma$ (say), corresponds to $ \sigma$. The other,
 $p_\mu$ (say),  is the endpoint of a lift of a leaf of $\mu^*$ and lies in the limit set of $\Gamma$.

We note that the basepoints which we need to consider for convergence of  Cannon-Thurston maps should lie on the algebraic locus and its pre-image under $h_n$.
If we try to connect this basepoint to a lift of the realisation of $\lambda_C$ by an arc in the right homotopy class, it might land at a point whose distance from $o_{\HHH^3}$ goes to $\infty$.
This is the point where the assumption that the cusp $P$ corresponding to $ \sigma$ is {\em untwisted} is relevant.
We shall explain this more precisely now.

Recall that the surface $S$ has a subsurface $B$ corresponding to an upper geometrically finite end.
We can choose a basepoint $o_{\HHH^2}$ in a component $\tilde B$ of the pre-image of $B$ in $\HHH^2$ so that on the other side of a component of $\Fr \widetilde B$, there lies a crown domain $C$ with tip $\zeta$.
We can assume, by perturbing $\Phi_n$ equivariantly that $\Phi_n(o_{\HHH^2}) = o_{\HHH^3}$ for all $n=1, 2, \dots, \infty$.
%For each $n$, we consider a pleated surface $f_n$ inducing $\rho_n$ between the fundamental groups, which realises $\lambda$, $\Fr B$, and some fixed pants decomposition of $B$.
%Let $\tilde \lambda$ be the pre-image of $\lambda$ in $\HHH^2$, and $\tilde f_n$ a lift of $f_n$ to a map from $\HHH^2$ to $\HHH^3$.
%For each boundary leaf $\ell$ of $\lambda$, consider an arc $\alpha_\ell$ connecting $o_B$ to $\ell$.
Let $\ell$ be a side of $C$ having $\zeta$ as its endpoint at infinity. Thus, $\ell$ is a lift of a component of $\lambda_C$.
%We connect $o_{\HHH^2}$ with $\ell$ by a geodesic arc $\alpha_\ell$.
Now $\lambda$ can be realised by a pleated surface, thought of as a map  from $S$ to $\HHH^3/G_n$ inducing $\rho_n$ at the level of  fundamental groups. Hence there is a realisation $\tilde \ell_n$ of $\ell$ in $\HHH^3$, given by a geodesic connecting the two endpoints of $\Phi_n(\ell)$.
%This lifts to an arc $\tilde \alpha_\ell$ connecting $o_{\HHH^2}$ to a lift $\tilde \ell$ of $\ell$ in $\HHH^2$.
%We can take such a basepoint on a pleated surface of the subsurface $B$, for it corresponds to a geometrically finite end.
%We shall now use the hypothesis that the cusp $P$ corresponding to $ \sigma$ is {\em untwisted}.

Recall that we assumed that  $P$ is conjoining and  untwisted.
Let $P_n$ be a Margulis tube in $\HHH^3/G_n$ converging to $P$ geometrically as $(G_n)$ converges geometrically to $\Gamma$.
Then in the model manifold $\mathbf M_n$, the vertical annulus  forming the part of $\partial P_n$ on the $B-$side  has bounded height as $n \rightarrow \infty$. This is because  its limit is a conjoining annulus.
Hence we can connect $o_{\HHH^3}$ with $\tilde \ell_n$ by a path $a^n_\ell$ \lq bridging over' a lift $\widetilde P_n$ of $P_n$ (i.e.\ the path $a^n_\ell$ travels up the bounded height lift $\widetilde P_n$ to move from $o_{\HHH^3}$ to $\tilde \ell_n$).

%
%
%Since $P$ is untwisted and $f_n|\Sigma$ converges to a realisation of $\mu$, the length of the geodesic arc $a^n_\ell$ homotopic to $\tilde f_n(\tilde \alpha_\ell)$ fixing the initial point and keeping the other endpoint on $\tilde f_n(\tilde \ell)$ has bounded length as $n \rightarrow \infty$.
%Let $\zeta$ be a crown-tip for $(\lambda, \sigma)$.
%Then there is a boundary leaf $\ell$ of $\lambda$ whose lift $\tilde \lambda$ has $\zeta$ as its endpoint at infinity.
The embedding $\Psi_n: \Sigma_+' \to \Sigma_-$ lifts to an embedding $\til{\Psi_n}: \til{\Sigma_+} \to  \til{\Sigma_-}$ of  universal covers. Both $\til{\Sigma_+},  \til{\Sigma_-}$ may be regarded as embedded in $\HHH^2$. Also,   $\ell$ lies in $\til{\Sigma_-}$.
Since $C$ is well approximated (by assumption),  $\til \Psi_n^{-1}(\ell)$ is realised in $\HHH^3$ by a limit of the $\til \ell_n$'s.
This implies that the geodesic $\til \ell_n$ passes at a bounded distance  from $\til P_n$.
Furthermore, since $P$ is untwisted, $\til \ell_n$ cannot move too far from $o_{\HHH^3}$ along $\til P_n$.
Therefore we can choose $a^n_\ell$ to have bounded length as $n \rightarrow \infty$.
%If $\ell$ is not entirely contained in $\tilde\Psi_n(\tilde \Sigma_+)$, then its projection to $\Sigma_+$ intersects a component $d_n$ of the frontier of $\Sigma_+$ in $S$.
%Since the frontier of $\Sigma_+$ corresponds to a cusp in $\HHH^3/\Gamma$ distinct from $\sigma$, the limit of $\tilde \ell_n$ is a geodesic having two distinct endpoints.
%Since $\Psi_n(\lambda)$ converges to $\mu$, we see that $\tilde \Psi_n(C)$ converges to a crown domain of $\mu$.
Thus we are in the situation of the previous paragraph whether or not $\ell$ is contained in $\til \Psi_n(\til \Sigma_+)$.
Since $a^n_\ell$ has bounded length, by defining its end-point to be the basepoint $o_\mu^n$ of the last paragraph, we see that $o_\mu^n$  is at a uniformly bounded distance from $o_{\HHH^3}$. Hence the image of $\zeta$ by the Cannon-Thurston map $c_n$ converges to an endpoint $p_\mu$ of a lift of $\mu^\ast$.
%This implies that the sequence $c_n$ of Cannon-Thurston maps sends the corresponding endpoint $x$ of the boundary leaf of $\lambda$
%to the translation of $p_\mu$ by bounded isometry. 
On the other hand, since $\zeta$ is a crown-tip,  $c_\infty(\zeta)$ coincides with 
 $p_\sigma$ by Theorem \ref{ptpre-ct}. Since
$p_\mu \neq p_\sigma$ as was shown above, we conclude  that $\lim_{n\to \infty} c_n(\zeta) \neq c_\infty (\zeta),$  establishing the \lq only if' part of Theorem \ref{pointwise convergence}, and  the \lq if' part of Theorem \ref{non-continuous}. $\Box$

\section{Pointwise convergence for points other than tips of crowns}\label{noncrown}
In this section, we shall prove that for any $\zeta \in S^1 (= \Lambda_G)$ that is not a crown-tip, 
$\{c_n (\zeta)\}$ converges to $c_\infty (\zeta)$, where the $c_n$ and $c_\infty$ denote  Cannon-Thurston maps as in Theorems \ref{pointwise convergence} and \ref{non-continuous}.
Our argument follows the broad scheme of   \cite[Section 5.5]{mahan-series2} but  is considerably more involved technically. In particular, we need to deal with several cases which did not arise in  \cite{mahan-series2}.

We consider the universal covering $p: \widetilde S\rightarrow S$, identifying $\widetilde S$ with $\HHH^2$ and the deck group with $\pi_1(S)$ as before.
Fix  basepoints $o_{\HHH^2} \in \HHH^2$ and $o_{\HHH^3} \in \HHH^3$ independent of $n$.
Let $\zeta$ be a point in $\Lambda_{G}$ (= $\partial G$ or $\partial_h G$ according as $S$ is closed or finite volume non-compact) such that $\zeta$ is not a crown-tip.
Let $r_\zeta: [0,\infty) \rightarrow \HHH^2$ be the geodesic ray from $o_{\HHH^2}$  to $\zeta$.
The representation $\rho_n$ induces a map  $\Phi_n: \HHH^2 \rightarrow \HHH^3$ such that $\Phi_n(\gamma x)=\rho_n(\gamma) \Phi_n(x)$.
By Proposition \ref{ptwisecrit}, it suffices   to show that the EPP condition  holds for $r_\zeta$.

\begin{comment}
\begin{specdef}[EPP condition]
\label{EPP}
There are functions $f_\zeta: [0,\infty) \rightarrow [0,\infty)$ with $f_\zeta(t) \rightarrow \infty$ as $t \rightarrow \infty$ and $M_\zeta: \naturals \rightarrow \naturals$ both depending on $n$ such that for any $N \in [0,\infty)$ and $s, t \geq N$, the geodesic arc in $\HHH^3$ connecting  $\Phi_n \circ r_\zeta (s_1)$ with $\Phi_n \circ r_\zeta(s_2)$ lies outside the $f_\zeta(N)$-ball centred at $o_{\HHH^3}$ for any  $n \geq M_\zeta(N)$.
\end{specdef}
\end{comment}

Let us briefly recall  the structure of ends of $(\HHH^3/G_\infty)_0$.
Take a relative compact core $C_\infty$ of $(\HHH^3/G_\infty)_0$.
Identify $C_\infty$ with $S \times [0,1]$ preserving the orientations.
Then $C_\infty \cap \Fr (\HHH^3/G_\infty)_0$ consists of annuli lying on $S \times \{0,1\}$ whose core curves are parabolic curves.
Let $\alpha_1, \dots , \alpha_p$ be parabolic curves  lying on $S \times \{1\}$. We call these upper parabolic curves. Let  $\alpha_{p+1}, \dots , \alpha_{p+q}$ be those lying on $S \times \{0\}$. We call these lower parabolic curves.
Identifying $S$ with $S \times \{0\}$ and $S \times \{1\}$, we may also regard these as  curves on $S$.
Recall that each component of $S \times\{1\} \setminus \bigcup_{j=1 \cdots p} \alpha_j$ faces  an upper end of $(\HHH^3/G_\infty)$ whereas each component of $S \times \{0\} \setminus  \bigcup_{j=p+1 \cdots p+q} \alpha_j$ faces a lower end.
Each of these ends is either geometrically finite or simply degenerate.
We let $\tilde e_1, \dots ,\tilde e_s$ be the upper simply degenerate ends and $\tilde f_1, \dots, \tilde f_t$ be the lower ones.

Take disjoint annular neighbourhoods $A_1, \dots , A_p$ of upper parabolic curves $\alpha_1, \dots , \alpha_p$ on $S \times \{1\}$ identified with $S$, and in the same way,  $A_{p+1}, \dots, A_{p+q}$ of lower parabolic curves $\alpha_{p+1}, \dots, \alpha_{p+q}$ on $S \times \{0\}$ identified with $S$.
We number the components of $S \setminus \cup_{j=1}^p A_j$  and $S \setminus \cup_{j=p+1}^{p+q} A_j$ so that components $\Sigma_1, \dots , \Sigma_s$ of $S \setminus \cup_{j=1}^p A_j$ correspond to simply degenerate ends $\tilde e_1, \dots , \tilde e_s$ respectively, and in $S \setminus \cup_{j=p+1}^{p+q} A_j$,   components $\Sigma_1', \dots , \Sigma'_t$ correspond to simply degenerate ends $\tilde f_1, \dots , \tilde f_t$ respectively. Then,
each  of $\Sigma_1, \dots , \Sigma_s; \Sigma'_1, \dots , \Sigma'_t$ supports the ending lamination of the corresponding simply degenerate end.
We denote the components of $S \setminus \cup_{j=1}^p A_j$ other than $\Sigma_1, \dots , \Sigma_s$ by $T_1, \dots , T_u$, and the components of $S \setminus \cup_{j=p+1}^{p+q} A_j$ other than $\Sigma_1' , \dots, \Sigma'_t$ by $T'_1, \dots ,T'_v$.
Each of $T_1, \dots , T_u; T_1', \dots, T'_v$ corresponds to a component of $\Omega_{G_\infty}/G_\infty$-- the surface at infinity of a geometrically finite end.

Let $q: \HHH^3/G_\infty \rightarrow \HHH^3/\Gamma$ be the covering map  induced by the inclusion of $G_\infty$ into $\Gamma$.
By Thurston's covering theorem \cite{thurstonnotes, canary-cover, OhQ}, each simply degenerate end of $(\HHH^3/G_\infty)_0$ has a neighbourhood that projects homeomorphically down to the geometric limit $(\HHH^3/\Gamma)_0$ under $q$.
We denote such neighbourhoods of $\tilde e_1, \dots, \tilde e_s$ by $\widetilde E_1, \dots , \widetilde E_s$, and those of $\tilde f_1, \dots , \tilde f_t$ by $\widetilde F_1, \dots, \widetilde F_t$.
Recall also that $\HHH^3/\Gamma$ has a  model manifold $\mathbf M_\Gamma$ which can be embedded in $S \times (0,1)$,  and we identify its non-cuspidal part $\mathbf M_\Gamma^0$ with $(\HHH^3/\Gamma)_0$ using the model map $f_\G$.
The corresponding ends of $(\HHH^3/\Gamma)_0$ in $S\times (0,1)$ and their neighbourhoods are denoted by the same symbols without tildes.
%Let us recall the form of Minsky's model manifold for $\HHH^3/G_\infty$.
%For our purpose, it is more convenient to regard it as being embedded in $S \times (0,1)$ and consider  a \lq brick decomposition' as in Section \ref{brickmfld}.
By moving the model manifold as in Section \ref{model geometric} preserving the combinatorial structure of  brick decomposition, 
% there is a bi-Lipschitz model manifold $\mathbf M_\Gamma$ and a bi-Lipschitz map $f_\Gamma: \HHH^3/\Gamma \to \mathbf M_\Gamma$, which is (topologically) an embedding into $S \times (0,1)$  and has 
we can assume that the model manifold $\mathbf M_\Gamma^0$, which is regarded as a subset in $S\times (0,1)$, and the model map $f_\Gamma$ have the following properties:
\begin{enumerate}[(i)]
\label{setting}
\item $\mathbf M_\Gamma^0$ is decomposed into \lq bricks' each of which is defined to be {\em the closure of a maximal family of parallel horizontal surfaces}.
We call such a decomposition into bricks the {\bf standard brick decomposition}.
In particular each brick has a form $\Sigma \times J$, where $\Sigma$ is an incompressible subsurface of $S$ and $J$ is a closed or  half-open interval. Further (Theorem \ref{thm_osa} (6)), the vertical boundary $\partial{\Sigma} \times J$ of a brick is contained in $\partial \mathbf M_\Gamma^0$. 
\item  There is a brick containing $S \times \{1/2\}$.
\item The image of the inclusion of $G_\infty$ into $\Gamma$ corresponds to the fundamental group carried by an incompressible immersion of $S$ into $\mathbf M_\Gamma^0$: the algebraic locus.
The surface consists of horizontal subsurfaces lying on $S \times \{1/2\}$ and annuli wrapping around torus boundary components.
\item We can assume that the brick containing $S\times \{1/2\}$ has the form $S\times [1/3,2/3]$.
Corresponding to simply degenerate ends $e_1, \dots, e_s$, there are bricks $\mathbf E_1=\Sigma_1 \times [2/3, 5/6), \dots , \mathbf E_s=\Sigma_s \times [2/3, 5/6)$ containing images of $e_1, \dots , e_s$ under $f_\Gamma^{-1}$.
\item We also have  bricks $T_1 \times [2/3, x_1], \dots , T_u \times [2/3, x_u]$  with $x_1, \dots , x_u \geq 5/6$.
\item Similarly, corresponding to simply degenerate ends $f_1, \dots , f_t$, there are bricks $\mathbf F_1=\Sigma'_1 \times (1/6, 1/3], \dots , \mathbf F_t=\Sigma'_t \times (1/6, 1/3]$, and $T_1' \times [x'_1, 1/3], \dots , T'_v \times [x'_v, 1/3]$ with $x'_1, \dots, x'_v \leq 1/6$.
\item Every end of $\mathbf M_\Gamma^0$ other than $f_\Gamma^{-1}(e_1), \dots , f_\Gamma^{-1}(e_s); f_\Gamma^{-1}(f_1), \dots , f_\Gamma^{-1}(f_t)$ lies at a horizontal level   in $(0, 1/6]$ or $(5/6, 1)$.
\item There are cusp neighbourhoods in $\mathbf M_\Gamma$ containing $A_1 \times [2/3, 5/6), \dots A_p \times [2/3, 5/6); A_{p+1} \times (1/6, 1/3], \dots , A_{p+q} \times (1/6, 1/3]$, which we denote respectively by $U_1, \dots, U_p; U_{p+1}, \dots , U_{p+q}$.
\end{enumerate}

The proof of the EPP condition which is required in order to establish that $c_n (\zeta) \to c_\infty (\zeta)$ splits into two cases. 
%Let $o \in \til{S}$ be a preferred base-point.
\begin{enumerate}[({\text Case} I:)]
	
	\item The geodesic ray $r_\zeta\colon [0,\infty) \to \HHH^2$ that connects the basepoint $o_{\HHH^2}$ to the point at infinity $\zeta$ is projected down to $S$ as a geodesic ray $\bar r_\zeta$ which enters and leaves each of the subsurfaces $\Sigma_i, \Sigma_i', T_i, T_i'$ infinitely often.
	\item The geodesic ray $\bar r_\zeta$ on $S$ eventually lies inside one of the subsurfaces $\Sigma_i, \Sigma_i', T_i, T_i'$.
\end{enumerate}

\subsection{Case I: Infinite Electric Length}
\label{case IEL}
We shall first consider Case I above and show that the EPP condition is satisfied.
We can assume that at least one of $s, t$ is positive since otherwise $G_\infty$ is geometrically finite and this case has already been dealt with in \cite{mahan-series1}.

Consider the pre-images $p^{-1}(\Sigma_1), \dots , p^{-1}(\Sigma_s)$ in $\hyperbolic^2$ of $\Sigma_1, \dots , \Sigma_s$.  Its union is denoted by $\widetilde \Sigma$.
In the same way, we denote by $\widetilde \Sigma'$ the union of the pre-images of $\Sigma'_1, \dots , \Sigma'_t$.
Equip $\HHH^2$ with an  electric metric $d_\Sigma$, by electrocuting  the components of $\widetilde \Sigma$.
Similarly, equip $\HHH^2$ with a different  electric metric $d_{\Sigma'}$, by electrocuting  the components of $\widetilde \Sigma'$.
The hypothesis of Case I is equivalent to the following.
%Next we consider the electric metric on $\HHH^3$ on which $G_i$ acts using ends in $(\HHH^3/\Gamma)_0$.
%Recall that we have simply degenerate algebraic ends $e_1, \dots, e_s; f_1, \dots , f_t$ and their quasi-convex neighbourhoods 

\begin{assumption}[Infinite electric length]
\label{iel}
We assume (for the purposes of this subsection) that the lengths of $r_\zeta$ with respect to $d_\Sigma$ and $d_{\Sigma'}$ are both infinite.
In this case, we say that $\zeta$ satisfies the {\bf IEL} condition.
\end{assumption}

Under Assumption \ref{iel}, our argument is similar to that in \S 5.3.3 of \cite{mahan-series2}.
Due to the IEL condition,  $r_\zeta$ has infinite length for both $d_\Sigma$ and $d_{\Sigma'}$. Hence the ray
$r_\zeta$ goes in and out of components of $\widetilde \Sigma$ (as also those of $\widetilde \Sigma'$) infinitely many times.
Since there is a positive constant bounding  from below the distance between any two disjoint components of $\widetilde \Sigma$ or $\widetilde \Sigma'$, we have the following.
(See \S 5.5.4 in  \cite{mahan-series2}.)
% the EPP condition can be verified using the electric metrics.
%We define the EPP condition for an electric metric $d'$ as follows.
%\begin{definition}
%Let $d'$ be an electric metric, either $d_\Sigma$ or $d_{\Sigma'}$.
%We say that $r_\zeta$ has EPP condition with respect to $d'$ if the condition of Definition \ref{} holds replacing 
%\end{definition}

\begin{lemma}
\label{electric EP}
Let $d'$ denote either $d_\Sigma$ or $d_{\Sigma'}$.
Then there exists a function $f_\zeta: \naturals \rightarrow \naturals$ with $f_\zeta(n) \rightarrow \infty$ as $n \rightarrow \infty$ such that if $t \geq N$, then $d'(r_\zeta(0), r_\zeta(t)) \geq f_\zeta(N)$.
\end{lemma}

%Recall that we have a $\rho_n$-equivariant map $\Phi_n : \HHH^2 \rightarrow \HHH^3$.
To prove the EPP  condition in Case I, we need to prove the following.
\begin{proposition}
\label{EPP infinite}
Suppose that $r_\zeta$ has infinite length in both $d_\Sigma$ and $d_{\Sigma'}$ as in Assumption \ref{iel}.
Then, the  EPP condition holds for $\zeta$.
\end{proposition}

The proof of Proposition \ref{EPP infinite} occupies the rest of this subsection.
Recall now that we have a model manifold $\mathbf M_\Gamma$ for the geometric limit $\HHH^3/\Gamma$ with a bi-Lipschitz  homeomorphism $f_\Gamma^{-1}: \HHH^3/\Gamma \to \mathbf M_\Gamma$, which is the inverse of the model map. Let $\tilde f_\Gamma^{-1}: \HHH^3 \to \widetilde{\mathbf M}_\Gamma$ denote its lift to
a map between the universal covers.
We denote the lift of the model metric on $\mathbf M_\Gamma$ to $\widetilde{\mathbf M}_\Gamma$ by $d_{\mathbf M}$.
Since $(\rho_n)$ converges algebraically to $\rho_\infty$, we have  a $\rho_\infty$-equivariant map $\Phi_\infty : \HHH^2 \rightarrow \HHH^3$.
The ray $\Phi_\infty \circ r_\zeta$  is projected (under the covering projection)  to a ray $r_\zeta^\Gamma$ in $\HHH^3/\Gamma$. 
Also, by composing  $\tilde f_\Gamma^{-1}$ with $\Phi_\infty \circ r_\zeta$, we get a  ray $r_\zeta^{\mathbf M}: [0,\infty) \rightarrow \widetilde{\mathbf M}_\Gamma$.\\

\noindent {\bf Idea of proof of Proposition \ref{EPP infinite}:}
To prove Proposition \ref{EPP infinite}, we shall consider the behaviour of a ray in  $\widetilde{\mathbf M}_\Gamma$ with respect to a new electric metric. The hypothesis of
Proposition \ref{EPP infinite} guarantees that the $d'$-diameter of the ray $[o_{\Hyp^2}, \zeta)$
is infinite. It thus enters and leaves lifts of some fixed subsurface $\Sigma$ (defined
as in the discussion preceding Assumption \ref{iel}) infinitely often. There is a brick (of the form $\Sigma \times J$)
containing  $\Sigma$  with boundaries on cusp neighborhoods
 of $\mathbf{M}_\Gamma$. If $\Sigma$ bounds a simply degenerate end, then the brick is simply its product neighbourhood.
 The boundary cusps corresponding to the boundary curves of $\Sigma$ 
 lifted to the universal cover $\widetilde{\mathbf{M}_\Gamma}$
  separate $\widetilde{\mathbf{M}_\Gamma}$. It would then
   suffice to show that
  a geodesic realisation of $[o_{\Hyp^2}, \zeta)$ enters and leaves such bricks infinitely often
  and does so in a such way that is compatible with the way $[o_{\Hyp^2}, \zeta)$
  enters and leaves lifts of  $\Sigma$ in ${\Hyp^2}$. More precisely, following the scheme
  of \cite{mahan-series2}, it suffices to show that the brick (of the form $\Sigma \times J$)
  has uniformly quasiconvex lifts to the universal cover.
  
  Unfortunately, this is not quite true. One obstruction is the problem of nesting: an end $\mathbf{E}$
  that $\Sigma$ faces need not be minimal and a proper 
  subsurface $\Sigma'$ of $\Sigma$ might face another end   $\mathbf{E}'$. This will simply force the
  brick $\mathbf{E}$ to be non-quasiconvex.
  To circumvent this difficulty, we first prove quasiconvexity for minimal ends (Lemma \ref{quasi-convex})
  and then proceed  to \lq second (to) minimal' ends and iterate  the argument for minimal ends in Lemma
  \ref{passing corresponding one} {\it after electrocuting the minimal ends}.
 An application of  Lemma \ref{ea-strong} at this stage guarantees the EPP condition needed to complete
 the proof of Proposition \ref{EPP infinite}.

Since the proof of Lemma  \ref{quasi-convex}, which is the starting point of the above scheme, is itself
quite involved, we give a brief idea of its proof here. We prove Lemma \ref{quasi-convex} by pulling back minimal ends in
$\mathbf{M}$  to the sequence $\mathbf{M}_n$ converging to it. This is a delicate operation as disjoint
ends of $\mathbf{M}$ might come together in the approximants $\mathbf{M}_n$. The first exercise here
is to thus identify the appropriate extension $\mathbf{E}^\mathrm{ext}$ of the minimal end
$\mathbf{E}$ of $\mathbf{M}$ for which such a phenomenon occurs. We shall
first describe the  construction of such an extended brick
below. However, such  bricks do not pull back to bricks in
the sequence $\mathbf{M}_n$. The problem here is caused by Margulis tubes that bound the pull back of
 $\mathbf{E}^\mathrm{ext}$ (and bricks are necessarily bounded by cusp neighbourhoods). To circumvent
 this difficulty, we apply the Brock-Bromberg drilling theorem \cite{BB} and establish
 a correspondence between bricks of the drilled manifold and those of $\mathbf{M}$. This finally
 allows us to reduce Lemma  \ref{quasi-convex} to proving quasi-convexity of bricks in the drilled manifold. This last step is reasonably standard once we realise that the bricks in the drilled
 manifold correspond to complements of cusps for a quasi-Fuchsian surface. We now proceed with the details of the above sketch.\\

\noindent {\bf Building the extended brick $\widehat{\mathbf E}^\mathrm{ext}$:}
Recall that we have bricks $\mathbf E_1, \dots ,\mathbf E_s; \mathbf F_1, \dots , \mathbf F_t$ in the non-cuspidal part $\mathbf M_\Gamma^0$ of the model manifold $\mathbf M_\Gamma$.
%We let $\mathbf E_1, \dots, \mathbf E_s$ denote the bricks $\Sigma_1 \times [2/3, 5/6) \dots , \Sigma_s \times [2/3, 5/6)$ in $\mathbf M_\infty^0$, and let $\mathbf F_1, \dots , \mathbf F_t$ denote the bricks $\Sigma'_1 \times (1/6, 1/3], \dots, \Sigma'_t \times (1/6, 1/3]$.
We say that a brick $\mathbf E$ in $\mathbf M_\Gamma^0$ corresponding to a neighbourhood $E$ of a simply degenerate end $e$ is {\bf minimal} if there is no other simply degenerate end whose neighbourhood is homotopic into $E$.
% $\Sigma'_1, \dots, \Sigma'_t$  can be isotoped into $\Sigma_j$.
%In the same way, we say that $\mathbf F_j$ among $\mathbf F_1, \dots , \mathbf F_t$ is minimal if none of $\Sigma_1, \dots, \Sigma_s$ can be isotoped into $\Sigma_j'$.
%%Let $\tilde{\mathbf E}_1, \dots , \tilde{\mathbf E}_s$ and $\tilde{\mathbf F}_1, \dots, \tilde{\mathbf F}_t$ the pre-images of $\mathbf E_1, \dots, \mathbf E_s$ and $\mathbf F_1, \dots, \mathbf F_t$ respectively in the universal cover $\tilde{\mathbf M}_\infty$.
%Let $U_1, \dots , U_{p+q}$ be the cusp neighbourhoods corresponding to $c_1, \dots , c_{p+q}$, and $\tilde U_1, \dots , \tilde U_{p+q}$ its pre-images in $\tilde{\mathbf M}_\infty$.

%Let $E$ be a neighbourhood of an end of $\mathbf M_\Gamma^0$,  having a form $\Sigma \times J$ where $J$ is a half-interval.
Let $\mathbf E$ be a brick  which  is not necessarily minimal, and suppose that $\mathbf E$ has the form $\Sigma \times J$ where $J$ is a half-interval.
Recall that a cusp neighbourhood $U$ in $\mathbf M_\Gamma$ is said to  abut on $E$ if the vertical boundary of $\mathbf E$ intersects $\partial U$.
A cusp neighbourhood $U$ is said to be {\bf associated with} $\mathbf E$ if the inverse image of  a horizontal curve on $\partial U$ under the approximate isometry  is homotopic into the inverse image of $\mathbf E$ in $\mathbf M_n$ for every large $n$.
We note that any cusp neighbourhood abutting on $\mathbf E$ is also associated with $\mathbf E$.

To make our description  easier, we now introduce the notion of $\sup$ and $\inf$ for cusp neighbourhoods and bricks.
Let $\pr_h \colon \mathbf M_\Gamma \to (0,1)$ be the horizontal projection to $(0,1)$ with $\mathbf M_\Gamma$ regarded as a subset of $S \times (0,1)$.
For a set $K$ in $\mathbf M_\Gamma$, we define $\sup K$ and $\inf K$ to be $\sup \pr_h(K)$ and $\inf \pr_h(K)$ respectively.

Suppose that $\mathbf E \cong \Sigma \times [2/3,5/6)$ is one of the  $\mathbf E_1, \dots  \mathbf E_s$ such that $\mathbf E$ is minimal.
Let $U_{l_1}, \dots , U_{l_\mu}$ be the cusp neighbourhoods abutting on $\mathbf E$.
We number all the cusp neighbourhoods associated with $\mathbf E$ (there might be infinitely many of them) as $U_{l_1}, \dots$, by extending $U_{l_1}, \dots , U_{l_\mu}$.
We let $\varsigma$ be $\min_{k=1}^\mu \sup(U_{l_k})$ respectively.
We note that $\varsigma \geq 5/6$ so that, in particular, $\mathbf E$ is contained in $\Sigma \times [2/3, \varsigma]$.

We renumber the $U_{l_k}$ with $k > \mu$ so that $U_{l_{\mu+1}}, \dots , U_{l_\nu}$ are all of those among the $U_{l_k}\ (k> \mu)$ that intersect $\Sigma \times \{\varsigma\}$.
(Note that every cusp neighbourhood intersecting $\Sigma \times \{2/3\}$ abuts on $\mathbf E$ since we assumed that $\mathbf E$ is minimal.
Also as was shown in  \S9.4.2 in \cite{OS}, we can assume, by choosing an appropriate embedding of $\mathbf M_\Gamma$ into $S \times (0,1)$, that every cusp neighbourhood associated with $\mathbf E$ whose $\sup$ is greater than $\varsigma$ must intersect $\Sigma \times \{\varsigma\}$.)
If $U_{l_k}$ is a $\integers$-cusp neighbourhood, we can extend it to a solid torus in $S \times (0,1)$ by adding the products of horizontal annuli and  closed or half-open vertical interval lying outside $\mathbf M_\Gamma$, as depicted as black rectangles in \cref{ext}.
We denote such an extension by $U_{l_k}^\mathrm{st}$.
Let $\mathbf M'_\Gamma$ be a brick manifold obtained by the standard brick decomposition of  $S \times (0,1) \setminus \cup_{k=1}^\nu U_{l_k}^\mathrm{st}$, i.e. the decomposition such that each brick is  the closure of maximal union of parallel horizontal subsurfaces as (i) in \cref{setting}.
Let ${\mathbf E}^\mathrm{br}$ be the union of all bricks in $\mathbf M'_\Gamma$ that are homotopic in $S \times (0,1)$ into $\Sigma \times [2/3,\varsigma]$.
Then we define $\widehat{\mathbf E}^\mathrm{ext}$ to be the intersection $\mathbf M_\Gamma \cap {\mathbf E}^\mathrm{br}$.
%If the upper horizontal boundary of a cusp neighbourhood lies outside $\mathbf M_\Gamma$, it must be on the same horizontal level of an end of $\mathbf M_\Gamma$.
%(This is a property of the embedding which we constructed in \cite{OS}.)
%Therefore the upper horizontal frontier of $\widehat{\mathbf E}^\mathrm{ext}$ is also a frontier component of $\dot{\mathbf E}$.

\begin{center}
\begin{figure}
\includegraphics[height=4cm]{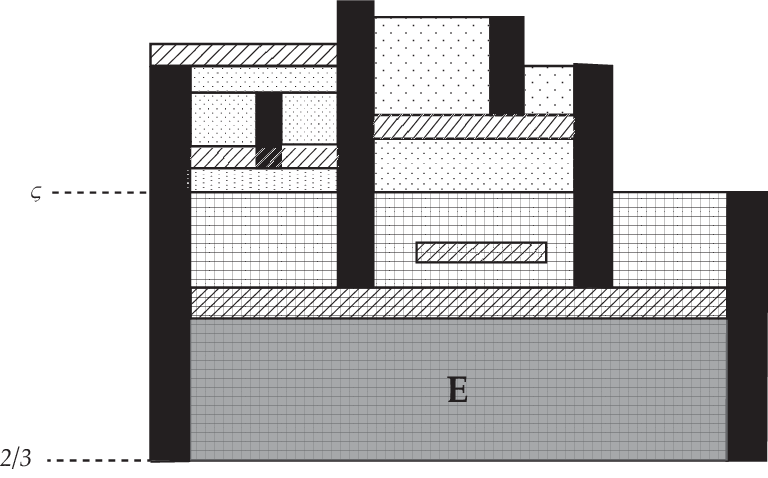}
\caption{The definition of $\widehat{\mathbf E}^\mathrm{ext}$.
Black rectangles are extended cusp neighbourhoods.
The grey transparent region is $\mathbf E$.
The diagonally lined regions lie outside $\mathbf M_\Gamma$.
The crossed region is $\Sigma \times [2/3,\varsigma] \cap \mathbf M_\Gamma'$.
The dotted regions are those that lie outside $\Sigma \times [2/3,\varsigma]$ but need to be added to construct $\widehat{\mathbf E}^\mathrm{ext}$.}
\label{ext}
\end{figure}
\end{center}

\noindent {\bf Pulling back the extended brick $\widehat{\mathbf E}^\mathrm{ext}$ to the sequence $\mathbf M_n$:}
Recall that corresponding to the geometric convergence of $\mathbf M_n$ to $\mathbf M_\Gamma$, there is an approximate isometry $\mathbf h_n\colon B_{R_n}(\mathbf M_n, \mathbf x_n) \rightarrow B_{K_n R_n}(\mathbf M_\Gamma, \mathbf x_\infty)$, and that $\mathbf h_n$ is a $K_n$-bi-Lipschitz homeomorphism, where $R_n \rightarrow \infty$, $K_n \rightarrow 1$, and $\mathbf x_\infty$ lies on the algebraic locus.
We should also recall, as was shown in \cite[Lemma 9.2]{OhD}, that the embeddings of $\mathbf M_\Gamma$ and $\mathbf M_n$ into $S \times (0,1)$ can be arranged so that the geometric convergence preserves the horizontal levels.

We shall now describe a particular complement $\mathbf M_n'$ of some Margulis tubes in $M_n$. 
Fix a positive constant $\epsilon_0$ less than the three-dimensional Margulis constant.
Then for large $n$, corresponding to cusp neighbourhoods $U_{l_1}, \dots , U_{l_\nu}$, each $\mathbf M_n$ has $\epsilon_0$-Margulis tubes converging to these cusp neighbourhoods.
We define $U^n_{l_k}$ to be the $\epsilon_0$-Margulis tube in $\mathbf M_n$ corresponding to $U_{l_k}$.
Then we define $\mathbf M_n'$ to be $\mathbf M_n \setminus \cup_{k=1}^\nu U^n_{l_k}$.
Since there is a homeomorphism between $\mathbf M_n'$  and $\mathbf M'_\Gamma$ preserving the horizontal levels, the union of bricks ${\mathbf E}^\mathrm{br}$ can be pulled back to a union of bricks in $\mathbf M_n'$, which we define to be $\widehat{\mathbf E}_n^\mathrm{ext}$. {\it We emphasise that $\widehat{\mathbf E}_n^\mathrm{ext}$ is a union of bricks in $\mathbf M_n'$, equipped with the standard brick decomposition, and not in $\mathbf M_n$. The difference arises precisely
	from the Margulis tubes $U^n_{l_k}$.}

We take $\epsilon >0$ small enough so that the $\epsilon$-neighbourhoods of $U_{l_1}, \dots$, which we denote by $U_{l_1}(\epsilon), \dots $ respectively, are disjoint.
Making $\epsilon$ smaller if necessary, we can assume that, for all $n$, the $\epsilon$-neighbourhoods  $U^n_{l_1}(\epsilon), \dots$ of $U^n_{l_1}, \dots$  are disjoint.
Now we define $\mathbf E^\mathrm{ext}$ to be $\widehat{\mathbf E}^\mathrm{ext} \setminus \cup_{k=1}^\nu U_{l_k}(\epsilon)$ and $\mathbf E^\mathrm{ext}_n$ to be $\widehat{\mathbf E}_n \setminus \cup_{k=1}^\nu U^n_{l_k}(\epsilon)$.
Then, by definition, $\{\mathbf E^\mathrm{ext}_n\}$ converges geometrically to $\mathbf E^\mathrm{ext}$ under the geometric convergence of $\mathbf M_n$ to $\mathbf M_\Gamma$.
We can define $\mathbf F^\mathrm{ext}$ and $\mathbf F_n^\mathrm{ext}$ when $\mathbf F$ among $\mathbf F_1, \dots, \mathbf F_t$ is minimal in the  same way.\\

\noindent {\bf Quasiconvexity of extended bricks $\widehat{\mathbf E}^\mathrm{ext}$:}
We shall need the following lemma, which is a generalisation of \cite[Lemma 2.25]{mahan-red}.

\begin{lemma}
\label{quasi-convex}
Let $\mathbf E$ be one of $\mathbf E_1, \dots , \mathbf E_s; \mathbf F_1, \dots , \mathbf F_t$, and assume that it is minimal.
%(We use the same symbols as those defined for $\mathbf E_j$, but without the subscript $j$ when it is attached.)
Let $U_{l_1}, \dots$ be the cusp neighbourhoods associated with  $\mathbf E$, and suppose that $U_{l_1}, \dots , U_{l_\nu}$ are those specified above.
We electrocute $\widetilde{\mathbf M}_\Gamma$ with respect to the components of the pre-images of $\sqcup_{k=1}^\nu U_{l_k}$, and obtain a Gromov hyperbolic metric $d^0_{\mathbf M}$.
In the same way, we electrocute $\widetilde{\mathbf M}_n$ at the components of the preimages of $\sqcup_{k=1}^\nu U^n_{l_k}$.
Then, there is a constant $K$ depending only on $\Gamma$ such that each component of the preimage $\widetilde{\mathbf E}^\mathrm{ext}$ of $\mathbf E^\mathrm{ext}$   in  $(\widetilde{\mathbf M}_\Gamma, d^0_{\mathbf M})$ and each component of the preimage  $\widetilde{\mathbf E}^\mathrm{ext}_n$ of $\mathbf E^\mathrm{ext}_n$ in $(\widetilde{\mathbf M}_n, d^0_n)$ are $K$-quasi-convex.
\end{lemma}

\begin{proof}
Since $\mathbf E^{\mathrm{ext}}$ is the geometric limit of $\mathbf E^\mathrm{ext}_n$ under the geometric convergence of $\mathbf M_n$ to $\mathbf M_\Gamma$, it suffices to prove that there exists $K \geq 1$
such that for all $n$,  each component of $\widetilde{\mathbf E}_n^{\mathrm{ext}}$ 
is $K$-quasi-convex in $\widetilde{\mathbf M}_n$.

Each horizontal boundary component of $\mathbf E^\mathrm{ext}_n$ corresponds to a horizontal boundary component of $\widehat{\mathbf{E}}^\mathrm{ext}$ under $\mathbf h_n$.
Fix a such a horizontal boundary component $F$ of $\widehat{\mathbf E}^\mathrm{ext}$, and consider the corresponding horizontal boundary component $F^n$ of $\mathbf E_n^\mathrm{ext}$ for each $n$.
Let $B_n$ be the brick in $\mathbf E^\mathrm{ext}_n$ containing $F^n$.
The geometric limit of $\mathbf E^\mathrm{ext}_n$ with basepoint on $F^n$ is a brick in $\mathbf E^\mathrm{ext}$ containing $F$.
We shall only describe the situation in the case when $\mathbf E$ is among $\mathbf E_1, \dots , \mathbf E_s$, and hence $F$ is an upper boundary component.
The case when $\mathbf E$ is among $\mathbf F_1, \dots , \mathbf F_t$ and $F$ is a lower boundary component can be dealt with in the same way.

The model manifold $\mathbf M_n$ contains Margulis tubes which induce a decomposition of $\mathbf M_n$ into blocks as in Minsky \cite{minsky-elc1}.
(We are using a slightly non-standard definition of Margulis tube here. We simply fix a constant $\ep>0$ and by \lq Margulis tube', we
means a tubular neighbourhood of a closed geodesic whose length is uniformly bounded.)
Let $a_1, \dots , a_p$ be core curves of annuli which are the intersection of $F^n$ with these Margulis tubes. 
The lengths of the core curves $a_1,\dots, a_p$ may be greater than the three-dimensional Margulis constant, but are bounded from above by the Bers' constant depending only the topological type of $S$  \cite[p. 20]{minsky-elc1}. 
By taking $\ep$ large enough, we can ensure that the curves constitute a pants decomposition of $F^n$.\\

\noindent {\bf Drilling:} We now proceed to drill the Margulis tubes  $U^n_{l_1}, \dots, U^n_{l_\nu}$ of $\mathbf M_n$, i.e.  remove the core geodesics of $U^n_{l_1}, \dots, U^n_{l_\nu}$ and equip the resulting rank two cusp with a complete hyperbolic metric, while leaving all the end-invariants 
(ending laminations or conformal structures) of $M_n$ unchanged (see \cite{BB} for details on drilling).
Let $M_n^\mathrm{dr} $ denote the drilled manifold obtained by drilling $M_n$.

 The drilling theorem of Brock-Bromberg \cite{BB} shows that there exist positive constants $\delta_n$
 satisfying the following:
 
 \begin{enumerate}
 \item $\delta_n \to 0$ as $n \to \infty$.
 \item $M_n$ is $(1+\delta_n)$-bi-Lipschitz homeomorphic to  ${M}_n^\mathrm{dr}$ away from the Margulis tubes and cusps corresponding to $U^n_{l_1}, \dots, U^n_{l_\nu}$.
 \end{enumerate} 
The  model manifold ${\mathbf M}_n^\mathrm{dr}$ is a model manifold of the non-cuspidal part of $M_n^\mathrm{dr}$ and is obtained by replacing the Margulis tubes $U^n_{l_1}, \dots, U^n_{l_\nu}$ of $\mathbf M_n$ by their torus boundaries.
Topologically,  ${\mathbf M}_n^\mathrm{dr}$ is the same as $\mathbf M_n \setminus \cup_{k=1}^\nu U^n_{l_k}$.

Since the new boundary components of ${\mathbf M}_n^\mathrm{dr}$ are precisely the torus boundaries 
of $U^n_{l_1}, \dots , U^n_{l_\nu}$, the brick decomposition of ${\mathbf M}_n^\mathrm{dr}$ can be assumed to coincide with that of $\mathbf M_n \setminus \cup_{k=1}^\nu U^n_{l_k}$. Recall that
$B_n$ is the brick in $\mathbf E^\mathrm{ext}_n$ containing $F^n$.
Thus, there is a brick and a horizontal boundary in ${\mathbf M}_n^\mathrm{dr}$ corresponding to $B_n$ and $F^n$, which we denote respectively by ${B}_n^\mathrm{dr}$ and ${F}^{n,\mathrm{dr}}$.
It thus suffices  to prove the uniform quasi-convexity of ${B}_n^\mathrm{dr}$.\\

\noindent {\bf Quasi-Fuchsian cover of drilled manifold:}
Recall that $\mathbf E \cong \Sigma \times [2/3,5/6)$.
Consider the covering $\overline{\mathbf M}_n$ of ${\mathbf M}_n^\mathrm{dr}$ corresponding to $\pi_1(\Sigma)=\pi_1({B}_n^\mathrm{dr})$.
Then  $\overline{\mathbf M}_n$  is a (cover corresponding to a) quasi-Fuchsian group of  type $\Sigma$, where the boundary components of $\Sigma$ are taken to be parabolics.
The shortest pants decomposition of the upper convex core boundary of $\overline{\mathbf M}_n$  is projected down to ${\mathbf M}_n^\mathrm{dr}$ as a pants decomposition with length bounded independently of $n$.
It follows that it is within a uniformly bounded distance from the simplex spanned by $a_1, \dots , a_p$ in the curve complex of $F$.
Therefore, if there is no homotopy between the pleated surface realising $a_1, \dots , a_p$ in $\overline{\mathbf M}_n$ and the upper convex core boundary with uniformly bounded diameter,  then there must be a parabolic element in the geometric limit of $\overline M_n$ which corresponds to a non-peripheral curve of ${F}^{n,\mathrm{dr}}$.
This contradicts the assumption that $F$ is an upper boundary component of $\mathbf E^\mathrm{ext}$ and there is no parabolic curve associated with $\mathbf E$ above such a boundary component which is homotopic into $F$.
This implies that there does exist such a homotopy with a uniformly bounded diameter (or equivalently, the tracks of points in the homotopy have  uniformly bounded diameter).

The lower horizontal boundary component of $\mathbf E^\mathrm{ext}$, which we denote by $\Sigma_\Gamma$, is homeomorphic to $\Sigma$, and if we take the covering of $M_\Gamma$ associated with $\pi_1(\Sigma)$, this boundary faces a geometrically finite end.
This horizontal surface corresponds to the lower boundary component of $\mathbf E_n^\mathrm{ext}$, which we denote by $\Sigma_n$ for each $n$.
The corresponding lower end of $\overline M_n$ has a neighbourhood converging geometrically to a neighbourhood of this end.
Therefore, by the same argument as above, we see that there is a homotopy between the pleated surface realising the core curve of the Margulis tubes intersecting $\Sigma_n$ in $\overline{M}_n$ and the lower convex core boundary with a uniformly bounded diameter.\\
%As we explained at the beginning, the same holds also when $F$ is a lower horizontal boundary.

\noindent {\bf Completing the proof of quasiconvexity, Lemma \ref{quasi-convex}:}
Now, we prove the uniform quasi-convexity of $\widetilde{\mathbf E}^\mathrm{ext}_n$ by contradiction.
Suppose that there is a sequence of arcs $a_n$ in $\widetilde{\mathbf E}^\mathrm{ext}_n$ such that
\begin{enumerate}
\item  The geodesic arc $a^*_n$ is homotopic to $a_n$ relative to endpoints.
\item  There exists a point $p_n$ in $a^*_n$ whose distance from $\widetilde{\mathbf E}^\mathrm{ext}_n$ with respect to $d^0_n$ goes to $\infty$ as $n \rightarrow \infty$.
\end{enumerate}

We can assume that $\mathbf E_n^\mathrm{ext}$ and the projection of these arcs $a_n$ and $a_n^*$ into $\mathbf M_n$ all lie in ${\mathbf M}_n^\mathrm{dr}$ instead of $\mathbf M_n$ since ${\mathbf M}_n^\mathrm{dr}$ and $\mathbf M_n$ are uniformly bi-Lipschitz away from cusps and Margulis tubes.
We consider the case when $p_n$ lies above $\widetilde{\mathbf E}^\mathrm{ext}_n$ with respect to the parametrisation $\widetilde S \times (0,1)$, where $\widetilde S$ denotes the universal cover of $S$.
(When $p_n$ lies below $\widetilde{\mathbf E}^\mathrm{ext}_n$, the argument works exactly in the same way by considering $\Sigma$ instead of $F_1, \dots , F_1$ in the argument below.)
Let $F_1, \dots F_q$ be the upper horizontal boundary components of $\mathbf E^\mathrm{ext}_\Gamma$, and $F^n_1, \dots , F^n_q$ the corresponding horizontal boundary components of $\mathbf E^\mathrm{ext}_n$.
Then there must be a subarc $b_n^*$ of the projection of $a_n^*$ into ${\mathbf M}_n^\mathrm{dr}$ containing the projection $\bar p_n$ of $p_n$, whose endpoints lie on the union of $F^n_1, \dots , F^n_q$ and $\partial U^n_{l_1}, \dots \partial U^n_{l_\mu}$ and which is homotopic into $\cup_{k=1}^q F^n_k \cup \cup_{j=1}^\nu U^n_{l_j}(\epsilon)$.
(Note that we are now in ${\mathbf M}_n^\mathrm{dr}$, and hence $U^n_{l_1}, \dots , U^n_{l_\nu}$ are torus cusp neighbourhoods.)
See Figure \ref{arc}.

\begin{center}
\begin{figure}
\includegraphics[height=4cm]{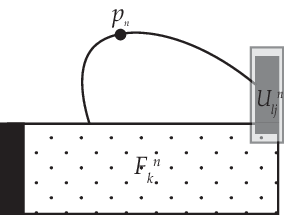}

\caption{The arc $b^*_n$ contains the point $p_n$, and is homotopic into  $\cup_{k=1}^q F^n_k \cup \cup_{j=1}^\nu U^n_{l_j}(\epsilon)$}
\label{arc}
\end{figure}
\end{center}
\begin{comment}
Is it better to itemise the conditions?
\end{comment}

The distance between $\bar p_n$ and $\mathbf E_n^\mathrm{ext}$ with respect to $d^0_n$ goes to $\infty$ by our choice of $p_n$.
On the other hand, $b_n^*$ must be contained in a uniformly bounded neighbourhood of the union of the convex cores of $\mathbf E_n^\mathrm{ext}$ and $\cup_{j=1}^\nu U^n_{l_j}$ since $b_n^*$ must be within a uniformly bounded distance from the corresponding electro-ambient geodesic (see Lemma \ref{ea-strong}) with respect to the electric metric obtained by electrocuting  components of the preimages of  $\cup_{j=1}^\nu U^n_{l_j}$ in the universal cover $\widetilde{\mathbf M}_n$.
Since there is a homotopy with uniformly bounded diameter between $F^n_k$ and the corresponding convex core boundary, this is a contradiction.
\end{proof}

\noindent {\bf The induction step:} We now explain the induction step mentioned in the outline of the proof of Proposition \ref{EPP infinite} following the statement of Proposition \ref{EPP infinite}.
This will culminate in Lemma \ref{passing corresponding one}.

It follows, as was explained in \S \ref{rel hyp},  that we can electrocute the components of  the pre-image $\widetilde{\mathbf E}^\mathrm{ext}$ of $\mathbf E^\mathrm{ext}$ in $(\widetilde{\mathbf M}_\Gamma, d_{\mathbf M}^0)$  and get a Gromov hyperbolic (pseudo-)metric on $\widetilde{\mathbf M}_\Gamma$. 
We can also electrocute the components of the preimage $\widetilde{\mathbf{E}}_n^\mathrm{ext}$ of $\mathbf E_n^\mathrm{ext}$ in $(\widetilde{\mathbf M}_n, d_n^0)$ and get a Gromov-hyperbolic space.
We denote the new electric metrics by $d^1_{\mathbf M}$ and $d^1_n$ respectively, and use 
 $\widetilde{\mathbf M}_\Gamma^1$ and $\widetilde{\mathbf M}_n^1$ as  shorthand for 
 $(\widetilde{\mathbf M}_\Gamma, d^1_{\mathbf M})$ and $(\widetilde{\mathbf M}_n, d^1_n)$. 
 Then $\widetilde{\mathbf M}_n^1$ converges geometrically to $\widetilde{\mathbf M}_\Gamma^1$.
 In the new metric,  geodesic arcs homotopic to subarcs of $r_\zeta^{\mathbf M}$ going deep   into $\widetilde{\mathbf E}$ may be conveniently ignored, which is what electrocution allows us to do.
However, we need to handle arcs that go deep into the pre-images of the other ends. For this, we need the second electrocution process as follows.

Let $\mathbf E'$ be one of $\mathbf E_1, \dots , \mathbf E_s; \mathbf F_1, \dots , \mathbf F_t$ other than the $\mathbf E$  we took in Lemma \ref{quasi-convex}.
We say that $\mathbf E'=\Sigma' \times [2/3, 5/6)$ among $\mathbf E_1, \dots , \mathbf E_s$ is second minimal if there is no $\Sigma_j'\, (j=1, \dots, t)$ that can be isotoped into $\Sigma$ except for the case when $\mathbf F_j$ is the $\mathbf E$  chosen above.
In the same way, we define $\mathbf E'= \Sigma \times (1/6, 1/3]$ among $\mathbf F_1, \dots, \mathbf F_t$ to be second minimal.

Suppose that $\mathbf E'$ is second minimal.
For simplicity of exposition, we assume that $\mathbf E'$ is one of $\mathbf E_1, \dots , \mathbf E_s$.
Below we define ${\mathbf E'}^\mathrm{ext}$ and ${\mathbf E'}_n^\mathrm{ext}$ in the same way as we defined $\mathbf E^\mathrm{ext}$ and $\mathbf E_n^\mathrm{ext}$:\\
By renumbering the $U_{l_k}$, let $U_{l_1}, \dots , U_{l_w}$ be the cusp neighbourhoods that abut on $\mathbf E'$.
We consider the brick $\dot{\mathbf E}'$ in $S \times (0,1) \setminus \cup_{k=1}^w U_{l_k}$ containing $\mathbf E'$,  let $s'$ be $\min \sup_{k=1}^w \sup U_{l_k}$, and let $U_{l_{w+1}}, \dots , U_{l_x}$ the cusp neighbourhoods that intersect $\Sigma' \times \{s'\}$.
Then we define $\hat{\mathbf E'}^\mathrm{ext}$ to be the intersection with $\mathbf M_\Gamma$ of the union of all bricks in $S \times (0,1) \setminus \cup_{k=1}^x U_{l_k}^\mathrm{st}$ that are homotopic into $\Sigma' \times [2/3, s']$. 
By deleting the $\epsilon$-neighbourhoods of $U_{l_1}, \dots , U_{l_x}$, we get ${\mathbf E'}^\mathrm{ext}$.
As before, by pulling back ${\mathbf E'}^\mathrm{ext}$ by $\mathbf h_n$, we get ${\mathbf E'}^\mathrm{ext}_n$.

%as follows.
%We first fix a positive number $\epsilon$ such that the $\epsilon$-neighbourhoods of $\integers$-cusps abutting on $\mathbf E'$ are disjoint, and if we delete the union of such neighbourhoods from $\mathbf E'$, we get a deformation retract of $\mathbf E'$.
%If $\epsilon$ is small enough, it has these properties, and we can choose $\epsilon$ depending only on $\Gamma$.
%Suppose that a $\integers$-cusp $U$ among $U_1, \dots , U_p$ is associated with $\mathbf E'$, i.e., abuts on $\mathbf E'$.
%Then, if $U$ is separating for $\mathbf E'$ and not contained in $\mathbf E^\mathrm{ext}$, we add it to $\mathbf E'$.
%If $U$ is either non-separating for $\mathbf E'$ or contained in $\mathbf E^\mathrm{ext}$, we delete the $\epsilon$-neighbourhood of $U$ from $\mathbf E'$.
%Next suppose that $U$ is a $\integers$-cusp neighbourhood among $U_{p+1}, \dots , U_{p+q}$  associated with $\mathbf E'$, i.e.\ it does not abut on $\mathbf E'$.
%If $U$ is contained in $\mathbf E^\mathrm{ext}$, we do nothing.
%Otherwise, we add $U \cup A \times [1/3, 2/3]$ to $\mathbf E'$, where $A$ denotes $A_j$ defined before if $U$ is $U_j$.

Let $\widetilde{\mathbf E'}^\mathrm{ext}$ and $\widetilde{\mathbf E'}^\mathrm{ext}_n$ be the pre-images of ${\mathbf E'}^\mathrm{ext}$ and ${\mathbf E'}^\mathrm{ext}_n$ in $\widetilde{\mathbf M}_\Gamma$ and $\widetilde{\mathbf M}_n$ respectively.
Then by nearly the same argument as in Lemma \ref{quasi-convex}, we can show that each component of $\widetilde{\mathbf E'}^\mathrm{ext}$ (resp. $\widetilde{\mathbf E'}^\mathrm{ext}_n$)  is $K'$-quasi-convex 
in $\widetilde{\mathbf M}_\Gamma^1$ (resp. $\widetilde{\mathbf M}_n^1$) after electrocuting the pre-images of cusp neighbourhoods of those among $U_{l_1}, \dots , U_{l_x}$ (resp.  Margulis tubes among $U^n_{l_1}, \dots , U^n_{l_x}$) that have not been electrocuted in the previous step.
Here $K'$ is a constant depending only on $\Gamma$.

The part where we need to modify the proof of Lemma \ref{quasi-convex} is the argument to deal with the case when the arc $b^*_n$ goes out from the lower horizontal boundary $\Sigma_n$. In
Lemma \ref{quasi-convex} we used the assumption that $\mathbf E$ is minimal.
In the present setting, it is possible that $\mathbf E$ is homotopic into $\mathbf E'$ in $\mathbf M_\Gamma$.
Still the argument involving the existence of homotopies with uniformly bounded diameters works since we have already electrocuted $\mathbf E^\mathrm{ext}$ and $\mathbf E_n^\mathrm{ext}$ and we can get homotopies with bounded diameter with respect to electrocuted metrics $d_n^1$.
%, and $\epsilon$ is also a constant chosen as above   depending only on $\Gamma$.
Therefore (c.f.\ the discussion in \S \ref{rel hyp}), we can again electrocute $\widetilde{\mathbf E'}^\mathrm{ext}$ in $(\widetilde{\mathbf M}_\Gamma, d^1_\mathbf{M})$  and $\widetilde{\mathbf E'}^\mathrm{ext}_n$ to get new Gromov hyperbolic metrics   denoted by $d^2_\mathbf{M}$ on $\mathbf M_\Gamma$ and  by $d^2_n$ on  $\mathbf M_n$ respectively .

We repeat this  procedure inductively.
Assume that we have defined the electric metrics $d^{m-1}_\mathbf{M}$ and $d^{m-1}_n$.
 Then at the next step of induction, we consider an $m$-th minimal  $\mathbf E$ among $\mathbf E_1, \dots , \mathbf E_s; \mathbf F_1, \dots, \mathbf F_t$. We construct $\mathbf E^{m\mathrm{ext}}$ and $\mathbf E^{m\mathrm{ext}}_n$ in the same way as we defined $\mathbf {E'}^{\mathrm{ext}}$ and ${\mathbf E'}^\mathrm{ext}_n$. 
 We define a new electric metric $d^m_\mathbf{M}$ and $d^m_n$, by electrocuting  the pre-images of $\mathbf E^{\mathrm{ext}}$ and ${\mathbf E'}^\mathrm{ext}_n$ together with preimages of suitable cusp neighbourhoods or Margulis tubes. The new metrics  are again Gromov hyperbolic.
Finally, we get  hyperbolic electric metrics $d^{s+t}_\mathbf{M}$ and $d^{s+t}_n$,  and denote them by $\bar d_\mathbf{M}$ and $\bar d_n$ respectively.

Recall that we have a  ray $r_\zeta^{\mathbf M}: [0,\infty) \rightarrow \widetilde{\mathbf M}_\Gamma$.
%This was defined to be the image of $\Phi_\infty \circ r_\zeta$ under $f_\infty^{-1} \circ$
For $s_1, s_2 \in [0,\infty)$ with $s_1< s_2$, we denote by $r_\zeta^{\mathbf M}(s_1, s_2)^*$ the geodesic arc, parametrised by length, homotopic to $r_\zeta^{\mathbf M}|[s_1, s_2]$ fixing the endpoints.
If we connect $r^\mathbf{M}_\zeta(s_1)$ and $r^\mathbf{M}_\zeta(s_2)$ by a geodesic arc with respect to $\bar d_{\mathbf M}$, it fellow-travels $r^\mathbf{M}_\zeta|[s_1, s_2]$ in $(\widetilde{\mathbf M}_\Gamma, \bar d_\mathbf{M})$ since all geometrically infinite ends  into which the geodesic may escape are electrocuted, along with the cusp neighbourhoods abutting on them.
This shows that the electro-ambient quasi-geodesic homotopic to $r_\zeta^{\mathbf M}|[s_1, s_2]$ fixing the endpoints must pass through the corresponding components of the pre-images of $\mathbf E^{\mathrm{ext}}$ for the components of $p^{-1}(\Sigma_j)\, (j=1, \dots , s)$ and  $p^{-1}(\Sigma_j')\, (j=1, \dots t)$  that  $r_\zeta([s_1,s_2])$ intersects essentially.
%Let $\mathbf o_\mathbf M$ be the basepoint of $\tilde{\mathbf M}_\Gamma$ which is $\tilde{f}_\infty^{-1}(o_{\hyperbolic^3})$.
Therefore, by Lemma \ref{ea-strong}, we have the following.

\begin{lemma}
\label{passing corresponding one}
There is a constant $L$ depending only on $\Gamma$ with the following property.
%Let $\mathbf E$ be either $\mathbf E_j$ or $\mathbf F_j$.
Let $\{R_\alpha\}_\alpha$ be the collection of all components of  $p^{-1}(\Sigma_j)\, (j=1, \dots , s)$ and  $p^{-1}(\Sigma_j')\, (j=1, \dots t)$  that  $r_\zeta([s_1,s_2])$ intersects essentially (relative to the endpoints).
We consider $\mathbf E^{\mathrm{ext}}$ for every $\mathbf E$ among $\mathbf E_1, \dots , \mathbf E_s; \mathbf F_1, \dots, \mathbf F_t$, and let $\widetilde{E}_{R_\alpha}$ be a component of the pre-image of one of them  corresponding to  $R_\alpha$.
Then $r_\zeta^{\mathbf M}(s_1, s_2)^*$ is contained in the $L$-neighbourhood of $\cup_\alpha \widetilde{E}_{R_\alpha}$ with respect to the metric $d_\mathbf{M}$.
\end{lemma}

We also have a $\rho_n$-equivariant map $\Phi_n \colon \HHH^2 \rightarrow \HHH^3$, and  a ray $\Phi_n \circ r_\zeta$ in $\HHH^3$.
By pulling back this ray by the lift of the model map $f_n$, we get a ray $r^n_\zeta \colon [0,\infty) \rightarrow \widetilde{\mathbf M}_n$, where $\widetilde{\mathbf M}_n$ is the universal cover of $\mathbf M_n$.
For $s_1<s_2 \in [0,\infty)$, we let $r^n_\zeta(s_1,s_2)^*$ be the geodesic arc in $\widetilde{\mathbf M}_n$ connecting $r^n_\zeta(s_1)$ and $r^n_\zeta(s_2)$.
Recall also that each component $R$ of $p^{-1}(\Sigma_j)$ or $p^{-1}(\Sigma_j')$ corresponds to a component of the pre-image of $\mathbf E_n^\mathrm{ext}$ for $\mathbf E=\mathbf E_j$ or $\mathbf E=\mathbf F_j$.
%Each component of the pre-image of $\mathbf E_j^\mathrm{ext}$ or $\mathbf F_j^\mathrm{ext}$ in turn corresponds to one of $\mathbf E_j^{n\mathrm{ext}}$ or $\mathbf F_j^{n\mathrm{ext}}$ in $\widetilde{\mathbf M}_n$.
We denote the component of the preimage of $\mathbf E_n^\mathrm{ext}$ corresponding to $R$ by $\widetilde E_R^n$.
Then, using  $\mathbf h_n$ to pull back the components appearing in  Lemma \ref{passing corresponding one}, we get the following.

\begin{lemma}
\label{uniformly farther}
There is a constant $L'$ independent of $n$ with the following property.
Let $\{R_\alpha\}_\alpha$ be the collection of all components of  $p^{-1}(\Sigma_j)\, (j=1, \dots , s)$ and $p^{-1}(\Sigma_j')\, (j=1, \dots t)$  that  $r_\zeta([s_1,s_2])$ intersects essentially (relative to the endpoints) as in Lemma \ref{passing corresponding one}.
%Let $\mathbf E^n$ be either $\mathbf E^n_j$ or $\mathbf F^n_j$.
%Let $\{R_\alpha\}_\alpha$ be the collection of the components of $p^{-1}(\Sigma_j)$ or $p^{-1}(\Sigma_j')$ (for $\Sigma_j$ or $\Sigma'_j$ corresponding to $\mathbf E^n$)  which $r_\zeta([s_1,s_2])$ intersects  essentially (relative to the endpoints).
Then $r^n_\zeta(s_1, s_2)^*$ is contained in the $L'$-neighbourhood of $\cup_\alpha \widetilde{E}^n_{R_\alpha}$ with respect to the metric $\bar d_n$.
\end{lemma}

\noindent {\bf Proof of Proposition \ref{EPP infinite}:}\\
Since the action of $\pi_1(S)$ on $\HHH^2$ corresponding to $G$ is properly discontinuous, Assumption \ref{iel} implies that there is a function $g\colon [0,\infty) \to [0,\infty)$ with $g(x) \to \infty$ as $x \rightarrow \infty$ such that $r_\zeta|[s_1,s_2]$ intersects only components of $p^{-1}(\Sigma_j)$ that are at a distance greater than $g(s_1)$ from the origin $o_{\HHH^2}$.
We now fix lifts of the basepoints $\mathbf o_\Gamma$ in $\widetilde{\mathbf M}_\Gamma$ and $\mathbf o_n$ in $\widetilde{\mathbf M}_n$. These are lifts of $\mathbf x_\infty$ and $\mathbf x_n$ respectively.
%By  proper discontinuity of the action of $\Gamma$ on $(\widetilde{\mathbf M}_\Gamma, \bar d_\Gamma)$, 
By the IEL assumption of the present Case I and since the electrocution process has been chosen so that two electrocuted pieces are separated by a minimum distance $\epsilon>0$, we  see that there is a function $k_\Gamma\colon [0,\infty) \rightarrow [0,\infty)$ with $k_\Gamma(x) \to \infty$ as $x\to \infty$ such that if $R$ is a component of $p^{-1}(\Sigma_j)$ that $r_\zeta|[s_1, s_2]$ intersects, then $\widetilde E_R$ is at a distance greater than $k_\Gamma(s_1)$ from $\mathbf o_\Gamma$.
By pulling this back by $\mathbf h_n$ for large $n$, we see that there is a function $k : [0,\infty) \rightarrow [0,\infty)$ with $k \to \infty$ as $x\to \infty$ such that if $R$ is a component of $p^{-1}(\Sigma_j)$ that $r_\zeta|[s_1,s_2]$ intersects, then $\widetilde E^n_R$ is at the  $\bar d_n$-distance greater than $k(s_1)$ from $\mathbf o_n$.
By Lemma \ref{uniformly farther}, this implies that any point of the geodesic arc  $r_n(s_1, s_2)^*$ is at a distance greater than $k(s_1)-L'$ from $\mathbf o_n$.
Since $f_n\colon \mathbf M_n \to \HHH^3/G_n$ is a  bi-Lipschitz map whose Lipschitz constant can be chosen independently of $n$,
this concludes the proof of Proposition \ref{EPP infinite}.
$\Box$

\subsection{Case II: $\bar r_\zeta$ is eventually contained in one subsurface}
Now, we turn to Case II, i.e.\ we suppose that Assumption \ref{iel} does not hold.
Then eventually $r_\zeta(t)$ stays in one component $R$ of $p^{-1}(\Sigma_j)$ or $p^{-1}(\Sigma'_j)$ or $p^{-1}(T_j)$ or $p^{-1}(T_j')$.
We now assume that $R$ is a component of $p^{-1}(\Sigma_j)$.
We can argue in the same way also for the case when $R$ is a component of $p^{-1}(\Sigma_j')$ just turning $\mathbf M_\Gamma$ upside down, whereas for the case when $R$ is a component of $p^{-1}(T_j)$ or $p^{-1}(T_j')$ we need a little modification of the argument, which we shall mention at the end of this subsection.
We can also assume that {\em if $r_\zeta(t)$ is also eventually contained in a component of $p^{-1}(\Sigma'_k)$, then $\Sigma'_k$ is {\em not} contained in $\Sigma_j$ up to isotopy, by choosing a minimal element among components containing $r_\zeta(t)$ eventually}.

\subsubsection{Case II A: when $\zeta$ is an endpoint of a lift of a boundary parabolic curve.}
We first consider the special case when $\zeta$ is an endpoint of a lift $\tilde c$ of a component $c$ of $\Fr \Sigma_j$.
Then $c_\infty(\zeta)$ is a parabolic fixed point of $\rho_\infty(\gamma_c)$, where $\gamma_c$ is an element of $G$ corresponding to $c$.
We consider the  geodesic axis $a_c$ of $\rho_n(\gamma_c)$, and let  $\mathcal N_n$ be a neighbourhood in $\HHH^3$ given by a lift of the $\epsilon_0$-Margulis around the projection of $a_c$ to $M_n$ for large $n$.
Correspondingly, there is a horoball $\mathcal N_\infty$ stabilised by $\rho_\infty(\gamma_c)$ to which $(\mathcal N_n)$ converges geometrically.
Since there is an upper bound for the distance from $\mathcal N_\infty$  to 
any point on $\Phi_\infty \circ \tilde c$, and hence $\Phi_\infty \circ r_\zeta$, there is an upper bound independent of $n$ for the distance from any point on the image of  $\Phi_n \circ r_\zeta$ to $\mathcal N_n$.

We define a broken geodesic arc $r_n^+(s_1,s_2)$  consisting of three geodesic arcs as follows. Let $\delta_{s_1}^n$ and $\delta_{s_2}^n$ be the shortest geodesic arcs that connect  $\Phi_n \circ r_\zeta(s_1)$ and $\Phi_n \circ r_\zeta(s_2)$ respectively with $\mathcal N_n$.  Let $\delta_{s_1,s_2}^n$ be the geodesic arc in $\mathcal N_n$ connecting the endpoint of $\delta_{s_1}$ on $\Fr \mathcal N_n$ with that of $\delta_{s_2}$.
We define $r_n^+(s_1,s_2)$ to be the concatenation $\delta^n_{s_1} * \delta^n_{s_1,s_2} * \delta^n_{s_2}$.
Then the observation in the previous paragraph implies that $r_n^+(s_1,s_2)$ is a uniform quasi-geodesic, i.e., there are constants $A,B$ independent of $s_1,s_2$ and $n$ such that $r_n^+(s_1,s_2)$ is an $(A,B)$-quasi-geodesic.

Now, by the convexity of Margulis tubes and the properness of $\Phi_n \circ r_\zeta$, it is easy to check that there is a function $g^+\colon [0,\infty) \to [0,\infty)$ with $g^+(u) \to \infty$ as $u \to \infty$ such that any point in $r_n^+(s_1,s_2)$ lies outside the $g^+(s_1)$-ball centred at $o_{\HHH^3}$.
Since $r_n^+(s_1,s_2)$ is uniformly quasi-geodesic, the geodesic arc connecting $\Phi_n \circ r_\zeta(s_1)$ with $\Phi_n \circ r_\zeta(s_2)$ is contained in a uniform neighbourhood of $r_n^+(s_1,s_2)$.
This shows the EPP condition for $r_\zeta$.

\subsubsection{Case II B: when $\zeta$ is neither an endpoint of the lift of a boundary parabolic curve nor a crown-tip}
Now,  we assume that $\zeta$ is neither an endpoint of a lift of a component of $\Fr \Sigma_j$ nor a crown-tip. Note that the latter is the standing assumption of this section.

Let $G^j$ be a subgroup of $G=\pi_1(S)$ corresponding to $\pi_1(\Sigma_j)$, and define $G_\infty^j$ to be $\rho_\infty(G^j)$.
The non-cuspidal part of the  hyperbolic 3-manifold $\HHH^3/G^j_\infty$ has a geometrically infinite end $e$ with a neighbourhood homeomorphic to $\Sigma_j \times (0,\infty)$.

The proof splits further into subcases.\\

\noindent {Subcase II B (i):}\\
We first prove the EPP condition for  the following special (sub)case. 
%A word about terminology. 
We say that the {\em geodesic realisation} of a geodesic (finite or infinite) in (the intrinsic metric on) $\til{S} (\subset \til{\mathbf{M_\G}})$ (resp. ${S} (\subset {\mathbf{M_\G}})$) is the geodesic in $\til{\mathbf{M_\G}}$ (resp. ${\mathbf{M_\G}}$) joining its end points  and path-homotopic to it. 

\begin{lemma}
\label{no smaller surface}
Suppose that there exists a $\Sigma_k'$   contained in $\Sigma_j$ (up to isotopy) and let $\mathbf{F_k}$ be the end corresponding to it.
If the geodesic  realisation $r_\zeta^{\mathbf M}$ in $\mathbf{M_\G}$ is not eventually disjoint from $\mathbf{F_k}$, then the EPP condition holds.
\end{lemma}
\begin{proof}
By our choice of $\Sigma_j$, it is impossible that the geodesic realisation of $r^{\mathbf M}_\zeta$ is eventually contained in one component of the pre-image of $\mathbf{F_k}$ in $\widetilde{\mathbf M}_\Gamma$, as this would imply that $p \circ r_\zeta$ is eventually contained in $\Sigma_k'$ contradicting our choice of $\Sigma_j$.
Therefore the only possibility under the hypothesis is that $r_\zeta^{\mathbf M}$ intersects infinitely many components of the pre-image of $\mathbf{F_k}$.
Then the argument  in Section \ref{case IEL} goes through to show the EPP condition.
\end{proof}

\noindent{Subcase II B (ii):}\\
Next we consider the subcase when there exists a $T'_k$  (among $T_1', \dots , T'_v$) contained in $\Sigma_j$ and $p \circ r_\zeta$ is {\em not} eventually disjoint from $T_k'$.

\begin{lemma}
\label{smaller T}
Suppose that there is a $T_k'$   contained in $\Sigma_j$ up to isotopy.
Let $U_1, \dots, U_q$ be the cusps abutting on the geometrically finite end of $(\HHH^3/\Gamma)_0$ corresponding to $T_k'$.
If the geodesic  realisation of the ray $r_\zeta^\Gamma$  is not eventually disjoint from the pre-images of $U_1\cup \dots \cup U_q$, then the EPP condition holds.
\end{lemma}
\begin{proof}
Under this assumption, the geodesic realisation of $r_\zeta^\Gamma$ intersects infinitely many horoballs that are lifts of $U_1, \dots, U_q$ since we are assuming that $\zeta$ is not an endpoint of a lift of a parabolic curve and hence that it is not eventually contained in one among the preimages of $U_1, \dots , U_q$. (Further, since $\zeta$ is not a crown-tip, it cannot be identified with the base-point of such a horoball either.)

We can assume that the same holds for the geodesic realisation of $r_\zeta^{\mathbf M}$ in the universal cover of the model manifold.
(We can properly homotope the ray through a bounded distance if necessary.)

By the approximate isometry $\mathbf h_n$, these cusps correspond to Margulis tubes $U^n_1, \dots , U_q^n$.
In this situation, we can electrocute $\widetilde{\mathbf M}_\Gamma$ at the pre-images of $U_1, \dots , U_q$ and $\widetilde{\mathbf E}^\mathrm{ext}$ together with the pre-images of the cusp neighbourhoods $U_{l_1}, \dots, U_{l_\nu}$ defined in Case \ref{case IEL}, where $\mathbf E$ is defined to be the brick $\mathbf E_j$.
In $\widetilde{\mathbf M}_n$, we  electrocute  
\begin{enumerate}
\item pre-images of the  Margulis tubes $U^n_1, \dots , U_q^n$,
\item  $\mathbf E^\mathrm{ext}_n$ and pre-images of Margulis tubes $U^n_{l_1}, \dots , U^n_{l_\nu}$.
\end{enumerate}  
Then by repeating the argument in Case \ref{case IEL},  the EPP condition follows.
\end{proof}

\noindent {Subcase II B (iii):}\\
Next we consider the (sub)case when there exists a $\Sigma'_k$ contained in $\Sigma_j$ up to isotopy, but $p  \circ r_\zeta$ is eventually disjoint from $\Sigma_k'$, and also there is no $T_k'$ as in Lemma \ref{smaller T}.

\begin{lemma}
\label{stay near}
Suppose that there exists at least one $\Sigma_k'$  contained in $\Sigma_j$ up to isotopy.
Suppose that $p \circ r_\zeta$ is eventually disjoint from any $\Sigma_k'$ that is contained in $\Sigma_j$, and that moreover it is not in the situation of Lemma \ref{smaller T}.
Then, there are  constants $R$ and $t_0$ independent of $n$ such that for all $s_1, s_2\in [0,\infty)$, both  greater than $t_0$, the geodesic arc $r_n(s_1,s_2)^*$ is contained in the $R$-neighbourhood of $\Phi_n(\hyperbolic^2)$.
\end{lemma}
\begin{proof}
Since $p\circ r_\zeta$ is eventually disjoint from any $\Sigma_k'$, there is $t_0$ such that $r_\zeta(t)$ is contained in a component $F$ of $\Sigma_j \setminus (\cup_k \Sigma'_k)$ if $t \geq t_0$.
Let $G^F$ be a subgroup of $G$ corresponding to $\pi_1(F)$. Let $G_n^F, G_\infty^F$  be the corresponding subgroups of $G_n$ and $G_\infty$ respectively. Let $\Gamma_F$ 
be the geometric limit of $(G_n^F)$.
Then by the covering theorem \cite{thurstonnotes,canary-cover,OhQ}, we see that $G^F_\infty$ is geometrically finite. Further, the only parabolics are those represented by the components of $\Fr \Sigma_k'$ and those of $\Fr T_l'$.
%Therefore $G_n^F=\rho_n(G^F)$ converges to $G_\infty^F$ strongly, and $\Gamma^F$ coincides with $G_\infty^F$.
Since $\Phi_n \circ r_\zeta$ converges to $r_\zeta^\Gamma$ under the geometric convergence of $G_n$ to $\Gamma$, we can assume that 
by taking larger $t_0$ if necessary,  $\Phi_n \circ r_\zeta(t)$ is contained in the $\epsilon$-neighbourhood of the  convex hull of the limit set of $G_n^F$ for $t \geq t_0$.
Furthermore, since we are not in the situation of Lemma \ref{smaller T}, the geodesic realisation of $r_\zeta^\Gamma$ is eventually disjoint from the horoballs corresponding to parabolic curves lying on $\Fr T_l$.
This implies that there is a constant $\delta$ such that  the geodesic realisation of $r_\zeta^\Gamma$ is within a bounded distance of the image $q \circ \Phi_\infty(\HHH^2)$.
Since $\Phi_n(\HHH^2)$ converges geometrically to $q \circ \Phi_\infty(\HHH^2)$, the geodesic arc connecting two points of $\Phi_n \circ r_\zeta(s_1)$ and $\Phi_n \circ r_\zeta(s_2)$ with $s_1, s_2 \geq t_0$ is also within uniformly bounded distance from $\Phi_n(\HHH^2)$.
%always contained in the $2\delta$-neighbourhood of the  boundary component of the Nielsen convex hull of $G_n^F$.
%Since the boundary component of the convex cores $\HHH^3/G_n^F$ corresponding to $F$ converges to that of $\HHH^3/G_\infty^F$ geometrically, and there is $R$ such that the pre-image of the convex core of $\HHH^3/\Gamma^F$ in $\HHH^3$ is contained in the $R'$-neighbourhood of the lift of $\Phi_\infty(\HHH^2)$, we can find $R$ as we wanted.
\end{proof}

The EPP condition in the situation of Lemma \ref{stay near} is now a replica of the proof of Theorems A, B in \cite{mahan-series1} as we are essentially reduced to the geometrically finite case.\\

\noindent {Subcase II B (iv):}\\ We have now come to the remaining subcase. In the discussion so far, we have already dealt with all the cases when the geodesic realisation of  $p_{\mathbf M} \circ r^{\mathbf M}_\zeta$ in $\mathbf{M}_\G$ can escape farther and farther from $\mathbf E$, where $p_{\mathbf M} \colon \widetilde{\mathbf M}_\Gamma \to \mathbf M_\Gamma$ is the universal covering.
Therefore,  we can assume that the geodesic  realisation of $p_{\mathbf M} \circ r^{\mathbf M}_\zeta$ in $\mathbf{M}_\G$ is  eventually contained {\em inside} a fixed simply degenerate brick  $\mathbf{E}$  of $\mathbf{M}_\G$.

We assume  that  $\mathbf{E}$  corresponds to $\Sigma_j$. 
The case when it corresponds to $\Sigma'_k$ can also be dealt with in the same way.
We can then, by moving  basepoints,  assume that $r_\zeta$ is entirely contained in $\Sigma_j$.
Let $G^j$ be the subgroup of $\pi_1(S)$ associated with $\pi_1(\Sigma_j)$. We shall use $G_n^j, G_\infty^j$ respectively for $\rho_n(G^j), \rho_\infty(G^j)$.
Let $\Gamma^j$ be a geometric limit of (a subsequence of) $(G_n^j)_{n \in \naturals}$. Then  $\Gamma^j$ can be regarded as a subgroup of $\Gamma$.

The geometric limit $\Gamma^j$ may be larger than $G_\infty^j$, but by the covering theorem (see \cite{canary-cover} and \cite{OhQ}), there is a neighbourhood $\bar E_j$ of $\bar e_j$, that projects homeomorphically to a neighbourhood $\hat E_j$ of a geometrically infinite end of $(\HHH^3/\Gamma^j)_0$.
Renumbering the parabolic curves on $S$, we assume that $c_1, \dots , c_p$ are the parabolic curves on $\Sigma_j$ including those corresponding to components of $\Fr \Sigma_j$.
Each  $c_k$ corresponds to either a $\integers$-cusp or a $\integers\times \integers$-cusp in $\HHH^3/\Gamma^j$. Let $U_k$ denote a neighbourhood of $c_k$.
Recall that by the geometric convergence of $\HHH^3/G_n^j$ to $\HHH^3/\Gamma^j$, these cusp neighbourhoods correspond to Margulis tubes in $\HHH^3/G^j_n$ for large $n$.
We denote the corresponding Margulis tubes in $\HHH^3/G_n^j$ by $U^n_1, \dots , U^n_p$.

We now electrocute $U_1, \dots, U_p$ in $\HHH^3/\Gamma^j$ to get a new metric $\bar d_{\Gamma^j}$, and accordingly we electrocute $U^n_1, \dots , U^n_p$ in $\HHH^3/G_n^j$ to get a new metric $\bar d_n$ in such a way that $(\HHH^3/G_n^j, \bar d_n)$ converges geometrically to $(\HHH^3/\Gamma^j, \bar d_{\Gamma^j})$.
Let $r_\zeta^\Gamma(s_1,s_2)^*$ be the geodesic arc in $\HHH^3$ connecting $r_\zeta^\Gamma(s_1)$ with $r_\zeta^\Gamma(s_2)$. Also, let $r_\zeta^\Gamma(s_1,s_2)^\times$ be the geodesic arc with respect to the electric metric $\bar d_{\Gamma^j}$. As per our previous terminology, $r_\zeta^\Gamma(s_1,s_2)^*$ is the geodesic realisation of $r_\zeta^\Gamma([s_1,s_2])$.
The hypothesis of Subcase II B (iv) can then be restated as follows:

\begin{comment}
Now, we shall show that geodesic arcs connecting two points of $r_\zeta^\Gamma$ cannot go out only from $E$ with respect to the metric $\bar d_\Gamma$.
Let $r_\zeta^\Gamma(s_1, s_2)^*$ be the geodesic arc in $\HHH^3$ connecting $r_\zeta^\Gamma(s_1)$ with $r_\zeta^\Gamma(s_2)$, and $r_\zeta^\Gamma(s_1,s_2)^\times$ be the one with respect to the electric metric $\bar d_\Gamma$.
\end{comment}
\begin{assumption}\label{ms2assume}
There is a constant $R'$ such that for any $s_1, s_2  \in [0,\infty)$, the geodesic realisation $r_\zeta^\Gamma(s_1, s_2)^*$  is contained in the $R'$-neighbourhood of $\hat E_j$ with respect to the (electric) metric $\bar d_\Gamma$.
\end{assumption}
\begin{comment}
Since $r_\zeta^\Gamma$ is the projection of $\Phi_\infty \circ r_\zeta$ by $q$, the geodesic arc $r_\zeta^\Gamma(s_1, s_2)^*$ is the image of the geodesic arc connecting $\Phi_\infty \circ r_\zeta(s_1)$ with $\Phi_\infty \circ r_\zeta(s_2)$, which we denote by $r_\zeta^\infty(s_1, s_2)$.
Since $(\HHH^3/G_\infty^j)_0$ has only one geometrically infinite end $\bar e_j$, we see that there is $R'$ such that $r_\zeta^\infty(s_1,s_2)^* \cap (\HHH^3/G_\infty^j)_0$ is contained in a $R'$-neighbourhood of $\bar E_j$.
Projecting this to the geometric limit $\HHH^3/\Gamma^j$, we see that $r_\zeta^\Gamma(s_1, s_2)^*$ is contained in the $R'$-neighbourhood of $\hat E_j$ with respect to $\bar d_\Gamma$.
Since $r_\zeta^\Gamma(s_1,s_2)^\times$ uniformly fellow-travels $r_\zeta^\Gamma(s_1, s_2)^*$ with respect to $\bar d_\Gamma$, we have thus completed the proof.
\end{proof}
\end{comment}

Under Assumption \ref{ms2assume}, we are reduced to the case where the geodesic realisations $r_\zeta^\Gamma(s_1,s_2)^*$ can only go deep into  the end  $\bar e^j$. This is exactly the situation dealt with in \cite[Corollary 6.13]{mahan-split} or in the proof of \cite[Theorem A]{mahan-series2} in Section 4.2.3 of that paper:\\
We approximate the geodesics $r_\zeta^\Gamma(s_1,s_2)^*$ or $r_\zeta^\Gamma(s_1, s_2)^\times$ by quasi-geodesics using the construction of model manifolds. Geometric convergence to the end $\mathbf{E}$ of the approximants $\mathbf{E}^n$ ensures the geometric convergence of their model manifolds. Corollary 6.13 \cite{mahan-split}  now translates to the EPP condition as in the proof of \cite[Theorem A]{mahan-series2}.

\subsubsection{The case when $p \circ r_\zeta$ is eventually contained in a geometrically finite $T_j$} Recall that $T_j$ corresponds to a geometrically finite end.
This case is simpler than previous one and we can repeat most of the arguments of the previous cases.
If $\zeta$ is an endpoint of a lift of a parabolic curve,  the argument in Case II A goes through without modification as
we did not use the assumption that $\Sigma_j$ corresponds to a simply degenerate end there.
For the analogues of Cases II B-(i, ii, iii) too we can argue in the same way as there to show the EPP condition.

In the remaining case, since the end $e$ corresponding to $T_j$ is geometrically finite, the geodesic realisation of $r_\zeta^\Gamma$ cannot escape towards $e$, as in the case of simply degenerate end.
Therefore the only possibility is that it lies in a neighbourhood of $q \circ \Phi_\infty(\HHH^2)$.
Thus, this is the same situation as Case II B (iii).

\section{Pointwise convergence for tips of crown domains}
\label{crowntips}
We shall now prove what remains in order to complete the proofs of Theorems \ref{pointwise convergence} and \ref{non-continuous}: \\
\noindent 
% If there is no coupled  geometrically infinite end  of $(\HHH^3/G_\infty)_0$ with well-approximated ending lamination which has an untwisted conjoining cusp abutting on it, then at the crown-tips, the sequence of Cannon-Thurston maps {\em do} converge. \\
If $\zeta$ is a tip of a crown domain $C$ for $(\lambda, \sigma)$, where $\sigma$ is a parabolic curve and $\lambda$ is an ending lamination,  and either the cusp corresponding to $\sigma$ is not conjoining, or $e$ is not coupled, or $\sigma$ is twisted, or $C$ is not well approximated, then the Cannon-Thurston maps {\em do} converge at $\zeta$.
 
 The proof splits into three cases:
 \begin{enumerate}
 	\item Either the cusp corresponding to $\sigma$ is not conjoining or $e$ is not coupled. This case will be dealt with in Proposition \ref{non-coupled} below, and for a special case in Proposition \ref{wellapprox}. 
	\item $e$ is coupled but the crown domain $C$ is {\it not} well approximated. 	This will be dealt with in Proposition \ref{wellapprox} below.
\item The cusp corresponding to $\sigma$ is twisted.
 	This will be dealt with in Proposition \ref{twistedcusp} below. 
 	 \end{enumerate}
 
 Now we consider the case when either the cusp corresponding to $\sigma$ is not conjoining or  $e$ is a simply degenerate end that is not coupled. We shall consider the point-wise convergence of Cannon-Thurston maps at crown-tips $(\lambda, \sigma)$.
  Without loss of generality, we assume that $e$ is upward as usual.
 %Then under the embedding $f: (\Hyp^3/\Gamma)_0 \to S \times (0,1) $ given by Theorem \ref{thm_osa} of the geometric limit $(\Gamma)$, the end $e$ satisfies  the following condition.
 Let $U$ be a $\integers$-cusp neighbourhood of $\HHH^3/\Gamma$ corresponding to $\sigma$.
 Let $A$ be an open annulus bounding $U$.
 Let $e'$ be another end on which $U$ abuts.
By assumption, $e'$ is not coupled with $e$.
 Then 
  \begin{enumerate}[(a)]
 \item Either $e'$ is also an upward algebraic end (if $U$ is not conjoining this is the only case), or 
 \item $e'$ is downward, but by taking the approximate isometry $h_n^{-1}$, a neighbourhood of $e'$ does not contain a surface $\Sigma'_+$ as in (3)-(c) of Definition \ref{def coupled} which is parallel into a neighbourhood of $e$.
 \end{enumerate}
% $f(E) $ is of the form $A \times (t,1)$, i.e. there is no part of the geometric limit `above' $E$. 
The proof of the case (b) is deferred and will be dealt with in Proposition \ref{wellapprox} together with the case when the crown domain $C$ is not well approximated.

% The open annulus $A$ has a core curve $\gamma$ which is regarded as a simple closed curve on $S$.
 We now state the point-wise continuity in the case (a) above.
 \begin{prop}\label{non-coupled}
 Let $e$ be a simply degenerate end of $(\HHH^3/G_\infty)_0$ corresponding to a subsurface $\Sigma$ of $S$, and let $\sigma$ be a component of $\Fr \Sigma$.
 Suppose that  either $e$ is not coupled or the cusp $U$ of $\HHH^3/\Gamma$ corresponding to $\sigma$ is not conjoining.
 In the former case, we assume moreover that   condition (a) above holds. Let $\lambda$ be the ending lamination for $e$, whose minimal supporting surface is $\Sigma$. Let $\zeta$ be a crown-tip of $(\sigma ,\lambda)$. Also, let $c_n\, (n=1, \cdots, \infty)$ denote the Cannon-Thurston maps for the representations $\rho_n\, (n=1, \cdots, \infty)$.
Then we have $c_n (\zeta) \to c_\infty (\zeta)$. 
 \end{prop}
 
 \begin{proof}
The cusp $U$ corresponds to a Margulis tube $U^n$ in $\HHH^3/G_n$ under the approximate isometry $h_n^{-1}$ for large $n$. The core curve of $U^n$ is freely homotopic to $\rho_n(\sigma)$.
We denote by $\mathbf U$ and $\mathbf U^n$ their counterparts in the model manifolds $\mathbf M_\Gamma$ and $\mathbf M_n$.
In the model manifold $\mathbf M_n$, the Margulis tube $\mathbf U^n$ is bounded by a torus $\mathbf T_n$ consisting of two horizontal annuli and two vertical annuli. These in turn correspond to the boundary $T_n$ of $U^n$.
Let $A_l^n, A_r^n$ be the two vertical annuli in $\mathbf T_n$, thought of as the left and right vertical annulus respectively.
Under condition (a), $\mathbf T_n$ converges to neither a torus nor a conjoining annulus.
Therefore, we see that the moduli of both $A_l^n$ and $A_r^n$ go to $\infty$ as $n \to \infty$.
Let $\mathbf E$ be the brick of $\mathbf M_\Gamma^0$ containing the end $f_\Gamma^{-1}(e)$.
We denote by $\mathcal U$ the union of cusp neighbourhoods of $\mathbf M_\Gamma$ abutting on $e$, and by $\mathcal U^n$ the union of Margulis tubes corresponding to $\mathcal U$ under the approximate isometry $\mathbf h_n$.
Then, if we regard $\mathbf M_n \setminus \mathcal U^n$ as a brick manifold with bricks consisting of maximal sets of parallel horizontal surfaces, there is a brick $\mathbf E_n$ of $\mathbf M_n \setminus \mathcal U^n$ which corresponds to $\mathbf E$ under $\mathbf h_n$.
Since $T_n$ converges to $A$ geometrically, we see that $\mathbf E_n$ converges geometrically to $\mathbf E$ whenever we choose a basepoint within a bounded distance from the algebraic locus.
We denote the counterpart of $\mathbf E_n$ and $\mathbf E$ in $\HHH^3/G_n$ and $\HHH^3/\Gamma$ by $E_n$ and $E$ respectively.
 
%Then, the upward end of the sequence of restrictions $\rho_n|_{\pi_1(A)}$ of $\rho_n$ to $\pi_1(A)$ converges geometrically to $E$. 
%Let $\gamma$ correspond to a curve bounding the crown-domain with $\zeta$ as a crown-tip. Then $\gamma$ corresponds to Margulis tubes in 
%$M_n = \HHH^3/\rho_n(\pi_1(S))$. Also, let $E_n$ be approximants of $E_n$ in $M_n$. We follow the convention that $E_n$ can intersect the Margulis tube $T_n$ corresponding to $\gamma$ in $M_n$ only along the boundary $\partial T_n$. The Margulis tube $T_n$ has two vertical annulus boundaries $A_l, A_r$, the left and right say, where the right boundary component $A_r$ is shared with $E_n$. Then by the hypothesis of Case 1, the `vertical lengths' of both $A_l, A_r$ tend to infinity as $n \to \infty$. To see this, first note that the vertical length of $A_r$ necessarily tends to infinity since the corresponding boundary annulus of $E$ is infinite. In the present situation the vertical length of $A_l$ also tends to infinity as otherwise the geometric limit would contain an end above $E$ which one could reach in (uniformly) bounded time for all $n$.

Now consider a pleated surface $f_n\colon S \to \HHH^3/G_n$  realising $\lambda$ and inducing $\rho_n$ between the fundamental groups. Let $\tilde f_n \colon \HHH^2 \to \HHH^3$ be its lift.
The endpoint at infinity $c_n(\zeta)$ is also an endpoint at infinity of a leaf $\ell_n$ of the pleating locus of $\tilde f_n$. It  is in fact a lift of a leaf of $\lambda_C$. Therefore, the geodesic realisation $r_\zeta^n$ of $\Phi_n \circ r_\zeta$ (setting its starting point to be $o_{\HHH^3}$) is asymptotic to this leaf $\ell_n$.
Next note that $(E_n)$ converges geometrically to $E$. Further, $\lambda_C$, which consists of leaves of $\lambda$, is not realisable in $E$.  Hence,
 for any compact set $K$, there exists $n_0$ such that for $n \geq n_0$, the surface $f_n(S) \cap E_n$ is disjoint from $h_n^{-1}(K)$.
Let $\widetilde U$ be a lift of $U$ corresponding to a lift of $\sigma$  lying on the boundary of the crown domain $C$ having $\zeta$ as a vertex.
Then $\widetilde U$ is a horoball touching $S^2_\infty$ at $c_\infty(\zeta)$.
The geodesic ray $r_\zeta^n$ can then be properly homotoped to a quasi-geodesic ray $\bar r_\zeta^n$ consisting of two geodesics; one connecting $o_{\HHH^3}$ to the point $z_n$ on $\ell_n$ nearest to $\widetilde U$, and a geodesic ray connecting $z_n$ to $c_n(\zeta)$ lying in the image of $\tilde f_n$.
Since the moduli of  both $A_l^n$ and $A_r^n$ go to $\infty$, the exterior angle of the two constituents of $\bar r_\zeta^n$ is bounded away from $\pi$. Hence $(\bar r_\zeta^n)$ is uniformly quasi-geodesic.
Since the time $\bar r_\zeta^n$ spends in $U^n$ goes to $\infty$, the same holds for $r_\zeta^n$.
Finally note that $U^n$ converges to $U$ and the latter lifts to a horoball $\widetilde U$ touching $S^2_\infty$ at only one point $c_\infty(\zeta)$. Hence $(c_n(\zeta))$ converges to $c_\infty(\zeta)$.
\end{proof}

Now we consider the second case when $P$ is conjoining but $e$ is not coupled or $C$ is not well approximated.

 \begin{prop}\label{wellapprox}
	Let $e$ be a simply degenerate end of $(\HHH^3/G_\infty)_0$ corresponding to the subsurface $\Sigma$ of $S$. Let $\lambda$ be the ending lamination for $e$ supported on $\Sigma$ and $\sigma$   a parabolic curve whose corresponding cusp abuts on  $e$. Suppose that 
	\begin{enumerate}
	\item Either $e$ is not coupled and the condition (b) before Proposition \ref{non-coupled} holds,
	\item Or, a crown domain $C$ for $(\lambda, 
	\sigma)$ is {\em not} well approximated. 
	\end{enumerate}
	Let $\zeta$ be a tip of $C$. Also, let $c_n\, (n=1, \cdots, \infty)$ denote the Cannon-Thurston maps for the representations $\rho_n\, (n=1, \cdots, \infty)$.
	Then $(c_n (\zeta))$ converges to  $c_\infty (\zeta)$. 
\end{prop}
\begin{proof}
We retain the notation in the proof of  Proposition \ref{non-coupled} and assume that $e$ is upward. 
We take a pleated surface $f_n\colon S \to \HHH^3/G_n$ realising $\lambda\cup \Fr \Sigma$.
Since the pleated surface $f_n$ realises $\lambda$, it follows that if $(f_n|\Sigma)$ converges geometrically, then after passing to a subsequence, we obtain a pleated surface map from $\Sigma$ to $\HHH^3/\Gamma$. This implies, however, that $e$ is coupled and $C$ is well approximated, contradicting the assumption.
Therefore, $(f_n)$ cannot converge geometrically.

Let $m_n$ be the hyperbolic structure on $\Sigma$ induced from $\HHH^3/G_n$ by $f_n$.
If $(m_n)$ is bounded in the moduli space then by the compactness of unmarked pleated surfaces, $(f_n)$ must converge geometrically  (after passing to a subsequence) as above.
Therefore $(m_n)$ is unbounded. Hence there is an essential simple closed curve $s_n$ in $\Sigma$ such that $\length_{m_n}(s_n) \to 0$.
We can choose  $s_n$ so that for a component $W_n$ of  $\Sigma \setminus s_n$ containing $\sigma$ in its boundary, $(W_n, m_n)$ converges geometrically to a complete hyperbolic surface.
If there is a limiting pleated surface realising $\lambda_C$ in the partner(s) of $e$, then $e$ is coupled and $C$ is well approximated, contradicting our assumption.

%On the other hand, since $(\rho_n)$ converges algebraically, the realised length of $\lambda$ by $f_n$ is bounded as $n \to \infty$.
%Therefore, in particular, the moduli of the hyperbolic structures induced by the $f_n$ are bounded.
It follows that the only possibility is that  the distance from the basepoint and $f_n(\Sigma)$ goes to $\infty$.
By repeating the argument in the proof of Proposition \ref{non-coupled}, we see that $(c_n(\zeta))$ converges to $c_\infty(\zeta)$.
\end{proof}
%
%
%%
%%Let $E_-$ denote the downward end coupled with $E$ above it. Then the ending lamination for $E$ coincides with that of $E_-$ as it is not well-approximated.
%
%Then $\Sigma_n$, the pleated surface realising $\lambda$ in $E_n \subset M_n$ has a subsurface corresponding to the subsurface $A$ of $S$ such that \\
%\begin{enumerate}
%	\item $\Sigma_n$ has a pleated subsurface $A_n$ corresponding to $A$, since the pleating locus includes the curve $\sigma$
%	\item The pleated subsurfaces $A_n$ exit the coupled end $E_-$ as $\lambda$ is {\bf not} well-approximated.
%\end{enumerate}
%
%Thus $[0, \zeta)_n$ corresponds to a diagonal or boundary leaf of the crown domain corresponding to the ending lamination $\lambda$ of $E_-$. 
%Hence $c_n (\zeta)$ corresponds to the terminal point of $[0, \zeta)_n$, which is realised as a geodesic in $E_n$ as an approximant of the downward end $E_-$.
%
%But by Theorem \ref{ptpre-ct}, the Cannon-Thurston map for $E_-$ identifies $\zeta$ with the base-point of the parabolic corresponding to $\sigma$, which in turn equals $c_\infty (\zeta)$. (Theorem \ref{ptpre-ct} from \cite{mahan-elct} proves this assertion for simply degenerate ends, but the same proof, viz. that of Proposition 3.1 of \cite{mahan-elct} establishes the same conclusion for a wild end). Hence $c_n (\zeta) \to c_\infty (\zeta)$. 

Now we turn to the last case when the cusp $P$ is twisted although the end is coupled, $P$ is conjoining, and $C$ is well approximated.
 
  \begin{prop}\label{twistedcusp}
 	Let $e$ be a coupled simply degenerate end of $(\HHH^3/G_\infty)_0$ corresponding to the subsurface $\Sigma$ of $S$. Let $\lambda$ be the ending lamination for $E$ supported on $\Sigma$ and  let $\sigma$ be a parabolic curve corresponding to a twisted cusp $P$ abutting on $e$. Let $\zeta$ be a tip of a crown domain $C$ of $(\lambda, c)$. We assume that $C$ is well approximated.
	Let $c_n\, (n=1, \cdots, \infty)$ denote the Cannon-Thurston maps for the representations $\rho_n$, $n=1, \cdots, \infty$.
 	Then $(c_n (\zeta))$ converges to $c_\infty (\zeta)$.
 \end{prop}
\begin{proof}
We shall use  notation from \S \ref{necessity}.
In particular, we consider a leaf $\ell$ of the pre-image of $\lambda$ in $\HHH^2$, whose endpoint at infinity is $\zeta$.
The only difference from the situation in \S \ref{necessity} is that the cusp $P$ is twisted now.
As in \S \ref{necessity}, we consider the realisation $\tilde \ell_n$ of $\ell$  and an arc $a^n_\ell$ in $\HHH^3$ connecting $o_{\HHH^3}$ to $\tilde \ell_n$ that bridges over $\widetilde P^n$.
Then, since $P$ is twisted, the landing point $y_n$ of $a^n_\ell$ has distance from $o_{\HHH^3}$ going to $\infty$, and the time  it spends in a lift $\widetilde P^n$ of the Margulis tube $P^n$ also goes to $\infty$.
Recall that $\widetilde P^n$ converges to a horoball $\widetilde P$ touching  $S^2_\infty$ at $c_\infty(\zeta)$.
This shows that $c_n(\zeta)$, which is  the endpoint at infinity of $\ell$, converges to $c_\infty(\zeta)$.
%We continue with the notation of Proposition \ref{non-coupled} above. The pleated surface $\Sigma_n \subset M_n$ is realised at bounded distance from the base point $o$. However, the twisting parameter along the curve on $S$ corresponding to $\sigma$ tends to infinity.
%Without loss of generality let the twisting parameter tend to $+ \infty$.
%To avoid burdening notation, we call this curve $\sigma$ as well.  Since $[o, \zeta)$ crosses $\sigma$ essentially, it follows that the geodesic realisations $[o, \zeta)_n$ in $M_n$ have the following properties:
%\begin{enumerate}
%	\item $[o, \zeta)_n$ enters the Margulis tube $T_n \subset M_n$ corresponding to $\sigma$ at a uniformly (independent of $n$) bounded distance from $o$.
%	\item Since the twisting parameter tends to infinity as $n \to \infty$, it follows that $[o, \zeta)_n$ exits $T_n$ at a distance $d_n$ from the entry point, where $d_n to + \infty$ as $n \to + \infty$.
%\end{enumerate}
%
%Let $\gamma_n = \rho_n (\sigma)$.
%By the second property above, the visual angle between $c_n (\zeta)$ and $\gamma_n^{+\infty}$ tends to zero as $n \to \infty$. But 
%$\gamma_n^{+\infty}$ converges to the base of the horoball corresponding to the parabolic $\sigma$ in $\rho_\infty (\pi_1(S))$, which in turn equals $c_\infty (\zeta)$. Hence $c_n (\zeta) \to c_\infty (\zeta)$.
\end{proof}

\section*{Acknowledgement} The authors would like to express their gratitude to the anonymous referee for his/her careful reading of the manuscript and valuable comments, which enabled us to improve our exposition.

\bibliographystyle{alpha}

\end{document}